\documentclass{amsart}
\usepackage[margin=1.2in]{geometry}

\usepackage{amsmath}%
\usepackage{amsfonts}%
\usepackage{amssymb}%
\usepackage{amsthm}
\usepackage{graphicx}
\usepackage{dsfont}
\usepackage{pdfcomment}
\usepackage{tikz}
\usepackage{tikz-3dplot}
\usepackage{subfigure}

\usepackage{hyperref}
\usepackage{cleveref}
\usepackage{esvect}

\usetikzlibrary{decorations.markings}
\usetikzlibrary{arrows}
\usetikzlibrary{matrix}

\newtheorem{theorem}{Theorem}[section]
\newtheorem{corollary}[theorem]{Corollary}

\newtheorem{definition}[theorem]{Definition}

\allowdisplaybreaks

\newtheorem{lemma}[theorem]{Lemma}

\newtheorem{proposition}[theorem]{Proposition}
\newtheorem{remark}[theorem]{Remark}

\DeclareMathOperator{\Arg}{Arg}

\DeclareMathOperator{\cratio}{cr}

\DeclareMathOperator{\Area}{Area}

\DeclareMathOperator{\Real}{Re}
\DeclareMathOperator{\Imaginary}{Im}

\renewcommand{\Re}{\Real}
\renewcommand{\Im}{\Imaginary}

\makeatletter
\newcommand\xleftrightarrow[2][]{%
	\ext@arrow 9999{\longleftrightarrowfill@}{#1}{#2}}
\newcommand\longleftrightarrowfill@{%
	\arrowfill@\leftarrow\relbar\rightarrow}
\makeatother

\title[Deformations of circle packings]{Minimal surfaces from infinitesimal deformations of \\circle packings}
\author{Wai Yeung Lam}

\address{Wai Yeung Lam\\
	Mathematics Department \\Brown University\\ Providence\\ RI 02912}

\email{lam@math.brown.edu}

\begin{document}
		
\begin{abstract}
	We study circle packings with the combinatorics of a triangulated disk in the plane and parametrize deformations of circle packings in terms of vertex rotation and cross ratios. We show that there is a Weierstrass representation formula relating infinitesimal deformations of circle packings to discrete minimal surfaces of Koebe type. Furthermore, every minimal surface of Koebe type can be extended naturally to a discrete minimal surface of general type. In this way, discrete minimal surfaces via Steiner's formula are unified.  
\end{abstract}
\maketitle
\section{Introduction}

William Thurston proposed circle packings as a discrete analogue of holomorphic functions \cite{Stephenson2005}. A circle packing is a configuration of circles with prescribed tangency patterns. By varying radii, there are many circle packings in the plane with the same tangency pattern. Motivated from the classical theory that holomorphic functions map infinitesimal circles to infinitesimal circles, Thurston suggest two circle packings with the same combinatorial structure exhibit a discrete holomorphic function. Using this idea, Rodin and Sullivan proved the convergence of circle packings to the Riemann mapping \cite{Rodin1987}. It results in a rich theory with many applications. For example, David Eppstein showed that every circle packing yields a configuration of planar soap bubbles \cite{Eppstein2014}. 

Further notions of discrete holomorphic functions have been proposed, such as circle patterns and vertex scaling. A circle pattern is a configuration of circles where neighboring circles intersect with prescribed angles \cite[Chap. 13]{Thurston1982}. In the case where all intersection angles are $\pi/2$,  it is called an orthogonal circle pattern \cite{Schramm1997}. Generally, given a triangle mesh in the plane, one can assign each triangular face with its circumscribed circle and obtain a circle pattern. Instead of circles, vertex scaling concerns the edge lengths of the triangle mesh \cite{Luo2004,Bobenko2010}. Both theories are closely related and can be expressed neatly in terms of complex cross ratios  \cite{Lam2015a}.          

An ongoing development for discrete holomorphic functions is their connection to the discrete surface theory \cite{Bobenko2006}. Particularly, it motivates from the Weierstrass representation in the smooth theory that every simply connected surface in space with vanishing mean curvature corresponds to a pair of holomorphic functions. It is interesting to obtain an analogous statement in the discrete theory, i.e. to obtain discrete minimal surfaces from discrete holomorphic functions.

However, similar to the case of discrete holomorphic functions, several definitions of discrete minimal surfaces have been proposed but their relation remains unclear.  One approach to define discrete minimal surfaces is via Steiner's formula, which relates the mean curvature of a surface to the area of its parallel surfaces. Suppose $f:U \subset \mathbb{R}^2 \to \mathbb{R}^3$ is a smooth surface with Gauss map $N:U \to \mathbb{S}^2$. For small $t$, we have parallel surfaces $f^t:= f +tN$ and Gauss map $N^t=N$. Steiner's formula states that the area 2-form of the parallel surfaces can be written as
\[
\langle f^t_x \times f^t_y, N \rangle dx \wedge dy = (1- 2H t + K t^2) \langle f_x \times f_y, N \rangle dx \wedge dy
\]
where $H$ and $K$ is the mean curvature and the Gaussian curvature of $f$. Following the smooth theory, Bobenko, Pottmann and Wallner \cite{Bobenko2010a} defined mean curvature for polyhedral surfaces $f:V \to \mathbb{R}^3$ equipped with certain vertex normals $N:V \to \mathbb{R}^3$ such that $f^t:= f + t N$ has planar faces with face normals unchanged. In this case, the area of $f^{t}$ on every face $\phi$ is written in the form
\begin{align}\label{eq:steiner}
\Area(f+tN)_{\phi} &=  (1- 2H t + K t^2)\Area(f)_{\phi}
\end{align} 
where $H$ and $K$ is defined as the mean curvature and the Gaussian curvature of $f$ on the face $\phi$ with respect to $N$. A discrete surface is minimal if its mean curvature $H$ vanishes. 

Nevertheless there are two problems in defining discrete minimal surfaces via Steiner's formula. Firstly, there is an ambiguity in the choice of $N$ \cite{Hoffmann2014}. Different choices of vertex normals lead to different classes of discrete minimal surfaces. In order to generalize previous examples from integrable systems, Bobenko, Pottmann and Wallner \cite{Bobenko2010a} considered three kinds of polyhedral surfaces $N$ as vertex normals:
\begin{itemize}
	\item Vertex offset: $N$ has vertices on the unit sphere and $N$ is parallel to $f$, i.e. $N_j - N_i \parallel f_j - f_i$ for every edge $ij \in E$. 
	\item Edge offset (\emph{Koebe type}): $N$ has edges tangent to the unit sphere and $N$ is parallel to $f$. In this case, $N$ is also called a Koebe polyhedral surface.
	\item Face offset: $N$ has faces tangent to the unit sphere and $N$ is parallel to $f$.
\end{itemize}
Secondly, it is unclear how to obtain these discrete minimal surfaces in a way analogous to Weierstrass representation formula from discrete holomorphic functions in general. Previous construction were restricted to quadrilateral meshes and special circle patterns \cite{Bobenko1996, Bobenko2006}. 

In order to unify the different classes of discrete minimal surfaces and relate them to discrete holomorphic functions, the author \cite{Lam2016} considered a generalization of discrete minimal surfaces:
\begin{definition}\label{def:genmin}
	Suppose $f:V \to \mathbb{R}^3$ is a polyhedral surface. We define its integrated mean curvature $H: F \to \mathbb{R}$ for every face $\phi$
	\begin{equation} \label{eq:meancur}
	\phi=(v_1,v_2,\dots, v_n, v_{n+1}=v_1) \in F \mapsto H_{\phi}= \sum_{j=1}^n \ell_{jj+1} \tan \frac{\alpha_{jj+1}}{2}
	\end{equation}
	where $\ell$ and $\alpha$ denote edge lengths and dihedral angles over edges. We say $f$ is a discrete minimal surface of general type if $H \equiv 0$.
\end{definition}
The quantity $\ell \tan \frac{\alpha}{2}$ is regarded as the principal curvature across edges, which vanishes if and only if the two neighboring faces are flattened. An important feature of the mean curvature formula $\eqref{eq:meancur}$ is that it is well defined for all polyhedral surfaces in space without referring to the choice of vertex normals. In the special case that a polyhedral surface admits face offsets, the formula \eqref{eq:meancur} coincides with the mean curvature in Steiner's formula \cite{Karpenkov2014}. 

In \cite{Lam2016,Lam2015a}, the author showed that there is a one-to-one correspondence between infinitesimal deformations of circle patterns and discrete minimal surfaces of general type via a Weierstrass representation formula. These discrete minimal surfaces include all the minimal surfaces via face offset and all the known examples via vertex offset. It remains a question if they include minimal surfaces via edge offset, which are called \emph{Koebe type} in \cite{Bobenko2017} (see Section \ref{sec:kobe}). The goal of this paper is to unify this remaining type of discrete minimal surface.

We first introduce discrete holomorphic quadratic differentials of Koebe type derived from infinitesimal deformations of circle packings  (Proposition \ref{prop:infmobkoebe}). Considering a circle packing in the plane with the combinatorics of a triangulated disk $G=(V,E,F)$, it is known that the circle packing admits non-trivial deformations with the same tangency pattern whenever $G$ has more than three boundary vertices. In particular, we call the change in cross ratios under an infinitesimal deformation of a circle packing as a discrete holomorphic quadratic differential (see Section \ref{sec:param}). It is analogous to the classical Schwarzian derivative that measures the degree how a conformal deformation fails to be a M\"{o}bius transformation.

We then establish a Weierstrass representation for discrete minimal surfaces of Koebe type in terms of discrete holomorphic quadratic differentials (Theorem \ref{thm:koebeweierstrass}). Together with Proposition \ref{prop:infmobkoebe}, it implies the correspondence between discrete minimal surfaces of Koebe type and infinitesimal deformations of circle packings.

Furthermore, we show that there is a natural construction to obtain a discrete minimal of general type from that of a Koebe type (Theorem \ref{thm:kobegeneral}), which leads to a unification of discrete minimal surfaces. It is known that every circle packing with the combinatorics of a triangle mesh admits a dual circle packing. The union of the two circle packings forms an orthogonal circle pattern, which is a configuration of circles where neighbouring circles intersect orthogonallly. An infinitesimal deformation of a circle packing then induces an infinitesimal deformation of the underlying circle pattern preserving the intersection angles. It has been shown in \cite{Lam2016} that the later yields a discrete minimal surface of general type. 

The main results are summarized in the following diagram.
	\begin{figure}[h!] \centering
	\begin{tikzpicture}[scale=0.6, every node/.style={scale=0.95}]
	\matrix (m) [matrix of math nodes,row sep=3em,column sep=4em,minimum width=2em]
	{
		\text{Infinitesimal deformations of circle packings} & \text{Infinitesimal deformations of circle patterns} \\
		\text{Holomorphic quadratic differentials of Koebe type} & \\
		\text{Minimal surfaces of Koebe type} & \text{Minimal surfaces of general type} \\};
	\path[-stealth]
	(m-1-1) edge  (m-2-1) 
	(m-1-1) edge  (m-1-2)
		(m-2-1) edge  (m-3-1) 
			(m-3-1) edge  (m-2-1) 
			(m-3-1) edge  (m-3-2) 
				(m-1-2) edge  (m-3-2) 
				(m-3-2) edge  (m-1-2)
	(m-2-1) edge (m-1-1);
	\path (m-1-1)--(m-2-1) node    [midway, left] {Proposition \ref{prop:infmobkoebe}};
	\path (m-2-1)--(m-3-1) node    [midway, left] {Theorem \ref{thm:koebeweierstrass}};
	\path (m-3-1)--(m-3-2) node    [midway, above] {Theorem \ref{thm:kobegeneral}};
	\path (m-1-2)--(m-3-2) node    [midway, right] {\cite{Lam2016}};
	\end{tikzpicture}
\end{figure}

Discrete minimal surfaces play an important role in structure-preserving discretization of differential geometry. Among many possible definitions of curvature for polyhedral surfaces, we are interested in those with rich mathematical structures and have connection to other established discrete theories. For example, this paper relates the integrated mean curvature formula (Eq. \eqref{eq:meancur}) to circle packings. Finding a proper definition of discrete minimal surfaces is a cornerstone to establish discrete analogues of classical differential geometry like Bernstein's theorem. 

In section \ref{sec:background}, we review notations and previous results on circle packing as well as discrete minimal surfaces. In section \ref{sec:param} we discuss parametrization of circle packings in terms of cross ratios. In section \ref{sec:verroation}, we consider infinitesimal deformations of circle packings and derive discrete holomorphic quadratic differentials. In section \ref{sec:kobe}, a Weierstrass representation formula is established to obtain discrete minimal surfaces of Koebe type from discrete holomorphic quadratic differentials. In section \ref{sec:unif}, we show that every discrete minimal surface of Koebe type can be extended naturally into that of general type. In section \ref{sec:har}, we purpose another parametrization of circle packings using vertex rotation.

 \section{Background} \label{sec:background}
 
 We consider a \textbf{circle packing} in the plane with the combinatorics of a triangulated disk $G=(V,E,F)$ with boundary. Each vertex $\mathfrak{u} \in V(G)$ is associated with a circle $C_{\mathfrak{u}}$ of radius $R_{\mathfrak{u}}$ centered at $c_{\mathfrak{u}}$ in such a way that for every edge $\mathfrak{uv} \in E(G)$, circles $C_{\mathfrak{u}}$ and $C_\mathfrak{v}$ touch at $z_i$ (Figure \ref{fig:orientation}). We denote $V_{int} \subset V$ and $E_{int} \subset E$ the set of interior vertices and interior edges. Furthermore we consider the set $\vec{E}$ of oriented edges. The edge oriented from $\mathfrak{u}$ to $\mathfrak{v}$ is written as $e_{\mathfrak{uv}} \in \vec{E}(G)$ and $e_{\mathfrak{uv}}\neq e_{\mathfrak{vu}}$. A (primal) 1-form of $G$ is a function on oriented edges $\vec{E}(G)$ such that $\omega(e_{\mathfrak{uv}}) = - \omega(e_{\mathfrak{vu}})$. A dual 1-form of $G$ is a 1-form on the dual graph.
 
 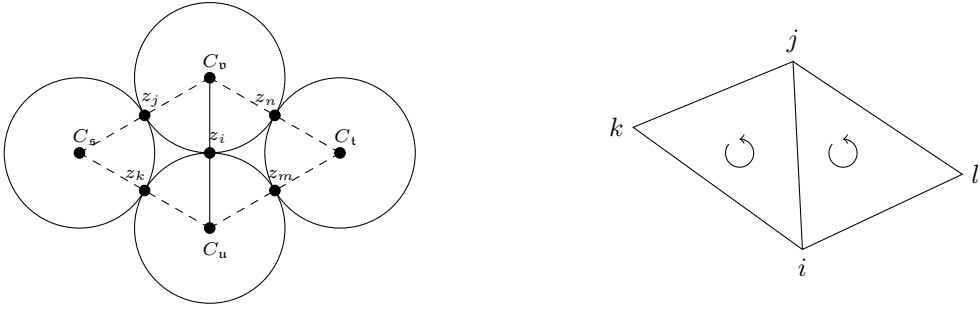
\begin{figure}[h!]
 	\centering

 	\begin{minipage}{0.45\textwidth}
 		\begin{tikzpicture}[line cap=round,line join=round,>=triangle 45,x=1.0cm,y=1.0cm]
 		\clip(-2.8,-2.1) rectangle (2.8,2.1);
 		\draw(0.,1.) circle (1.cm);
 		\draw(0.,-1.) circle (1.cm);
 		\draw[dash pattern=on 3pt off 3pt] (0.,1.)-- (-1.7320508075688772,0.);
 		\draw[dash pattern=on 3pt off 3pt] (-1.7320508075688772,0.)-- (0.,-1.);
 		\draw[dash pattern=on 3pt off 3pt] (0.,-1.)-- (1.7320508075688772,0.);
 		\draw[dash pattern=on 3pt off 3pt] (1.7320508075688772,0.)-- (0.,1.);
 		\draw(-1.7320508075688772,0.) circle (1.cm);
 		\draw(1.7320508075688772,0.) circle (1.cm);
 		\draw (0.,1.)-- (0.,0.);
 		\draw (0.,0.)-- (0.,-1.);
 		\begin{scriptsize}
 		\draw [fill=black] (0.,0.) circle (2.0pt);
 		\draw[color=black] (0.10306616478532194,0.20824815695886081) node {$z_i$};
 		\draw [fill=black] (0.,1.) circle (2.0pt);
 		\draw[color=black] (0.08306616478532194,1.213040216250935) node {$C_\mathfrak{v}$};
 		\draw [fill=black] (0.,-1.) circle (2.0pt);
 		\draw[color=black] (0.08306616478532194,-1.2965439023332133) node {$C_\mathfrak{u}$};
 		\draw [fill=black] (1.7320508075688772,0.) circle (2.0pt);
 		\draw[color=black] (1.8163324670641485,0.20824815695886081) node {$C_\mathfrak{t}$};
 		\draw [fill=black] (-1.7320508075688772,0.) circle (2.0pt);
 		\draw[color=black] (-1.6502001374935047,0.20824815695886081) node {$C_\mathfrak{s}$};
 		\draw [fill=black] (-0.8660254037844388,0.5) circle (2.0pt);
 		\draw[color=black] (-0.7835669863540914,0.7106441866048978) node {$z_j$};
 		\draw [fill=black] (-0.8660254037844386,-0.5) circle (2.0pt);
 		\draw[color=black] (-0.9835669863540914,-0.29414787268717624) node {$z_k$};
 		\draw [fill=black] (0.8660254037844388,-0.5) circle (2.0pt);
 		\draw[color=black] (0.9496993159247352,-0.29414787268717624) node {$z_m$};
 		\draw [fill=black] (0.8660254037844386,0.5) circle (2.0pt);
 		\draw[color=black] (0.7496993159247352,0.7106441866048978) node {$z_n$};
 		\end{scriptsize}
 		\end{tikzpicture}
 	\end{minipage}
  	\begin{minipage}{0.45\textwidth}
 	\centering
 	\begin{tikzpicture}[scale=1.25]
 	\coordinate [label=below:$i$] (i) at (0.1,-1);
 	\coordinate [label=above:$j$] (j) at (0,1);
 	\coordinate [label=left:$k$] (k)  at (-1.7,0.3);
 	\coordinate [label=right:$l$] (kt) at (1.8,-0.2);
 	\draw (i)--(j)--(kt)--(i);
 	\draw (j)--(k)--(i);
 	\draw [->] (0.8,-0.2) ++(140:5mm) arc (-220:90:1.5mm);
 	\draw [->] (-0.3,-0.2) ++(140:5mm) arc (-220:90:1.5mm);	
 	\end{tikzpicture}
 \end{minipage}
 	
 	\caption{Two neighboring triangles containing the edge $\mathfrak{uv}\in E(G)$ (left) and the edge $\{ij\} \in E(TMG)$ (right)}
 	\label{fig:orientation}
 \end{figure}
 
 \textbf{Circle patterns} are generalizations of circle packings where neighboring circles are allowed to intersect. Instead of the tangency of circles, two circle patterns have the same pattern structure if the corresponding intersection angles are the same. For a triangle mesh $z:V \to \mathbb{C}$ in the plane, a circle pattern is induced from the circumscribed circles of the faces. The intersection angles define a function on the interior edges. 
 
 A circle packing of the combinatorics of a triangle mesh induces an \textbf{orthogonal circle pattern} (Fig. \ref{fig:circlepacking}). For every three pairwise touching circles $C_{\mathfrak{u}}, C_{\mathfrak{v}}, C_{\mathfrak{s}}$, their tangency points $z_{i},z_{j},z_{k}$ lie on a circle that intersect the three circles orthogonally. These additional circles form a so-called dual circle packing. Combining a circle packing with its dual, we obtain an orthogonal circle pattern (Fig. \ref{fig:circlepacking} center). 
 
  \begin{figure}[h!]
 	\centering
 	\begin{minipage}{0.32 \textwidth}
 		\begin{tikzpicture}[line cap=round,line join=round,>=triangle 45,x=1.0cm,y=1.0cm,scale=0.7]
 		\clip(-3.,-2.8) rectangle (3.,2.8);
 		\draw(0.,0.) circle (1.cm);
 		\draw(2.,0.) circle (1.cm);
 		\draw(-2.,0.) circle (1.cm);
 		\draw(-1.,1.7320508075688772) circle (1.0000878721036468cm);
 		\draw(1.,1.7320508075688772) circle (1.0033111550690994cm);
 		\draw(3.,1.7320508075688772) circle (1.001878446163976cm);
 		\draw(-3.,1.7320508075688772) circle (1.000031327963911cm);
 		\draw(-1.,-1.7320508075688772) circle (1.000034097123565cm);
 		\draw(1.,-1.7320508075688772) circle (1.0009433592074812cm);
 		\draw(3.,-1.7320508075688772) circle (1.0034182825738462cm);
 		\draw(-3.,-1.7320508075688772) circle (1.0000040084275728cm);
 		\draw(-2.0064808217537324,3.4641016151377544) circle (1.0053548382577902cm);
 		\draw(0.0032482965229566887,3.4641016151377544) circle (0.9952221677733265cm);
 		\draw(1.997806592754283,3.4641016151377544) circle (0.9990427477938433cm);
 		\draw(-2.0055568446814505,-3.4641016151377544) circle (1.0027583052641529cm);
 		\draw(0.,-3.4641016151377544) circle (1.0000541647246564cm);
 		\draw(1.9835597477812021,-3.4641016151377544) circle (0.9918180278632107cm);
 		\begin{scriptsize}
 		\draw [fill=black] (1.,0.) circle (2.0pt);
 		\draw [fill=black] (-1.,0.) circle (2.0pt);
 		\draw [fill=black] (-3.,0.) circle (2.0pt);
 		\draw [fill=black] (3.,0.) circle (2.0pt);
 		\draw [fill=black] (-0.5080962528209986,0.8613002948271061) circle (2.0pt);
 		\draw [fill=black] (0.5490247793465535,0.8358060729998725) circle (2.0pt);
 		\draw [fill=black] (2.537073188919372,0.8435356481761612) circle (2.0pt);
 		\draw [fill=black] (-2.5048394513469736,0.8632132577548184) circle (2.0pt);
 		\draw [fill=black] (-0.5050484569517069,-0.8630909894852917) circle (2.0pt);
 		\draw [fill=black] (0.47316167858547153,-0.8809756102856531) circle (2.0pt);
 		\draw [fill=black] (2.448489301586677,-0.8937881999457673) circle (2.0pt);
 		\draw [fill=black] (-2.4982651225517767,-0.867024721474805) circle (2.0pt);
 		\draw [fill=black] (-1.5424626932845245,2.5722352982156043) circle (2.0pt);
 		\draw [fill=black] (0.48880317675156626,2.595364818683489) circle (2.0pt);
 		\draw [fill=black] (1.550871543308617,2.5706054072574562) circle (2.0pt);
 		\draw [fill=black] (-0.9940308515927064,3.4641016151377544) circle (2.0pt);
 		\draw [fill=black] (-0.4803706777031921,2.59428562827853) circle (2.0pt);
 		\draw [fill=black] (-0.0016143804893224978,1.6737249412394237) circle (2.0pt);
 		\draw [fill=black] (-2.000028273754884,1.7211327665528937) circle (2.0pt);
 		\draw [fill=black] (2.000718213249537,1.659965222624701) circle (2.0pt);
 		\draw [fill=black] (-2.0000150446346328,-1.7382238063650053) circle (2.0pt);
 		\draw [fill=black] (-4.548532329533561E-4,-1.7633189220657002) circle (2.0pt);
 		\draw [fill=black] (1.9987598396345276,-1.798129480942134) circle (2.0pt);
 		\draw [fill=black] (-1.4996165869778435,-2.598336823516471) circle (2.0pt);
 		\draw [fill=black] (-0.4918688111991666,-2.5933700807641284) circle (2.0pt);
 		\draw [fill=black] (1.4675578962670535,-2.6170803110205325) circle (2.0pt);
 		\draw [fill=black] (0.9992679650524392,3.4323680224197624) circle (2.0pt);
 		\draw [fill=black] (-0.9980613301393708,-3.5272040213007516) circle (2.0pt);
 		\draw [fill=black] (0.4724219809560179,-2.5826676673613367) circle (2.0pt);
 		\draw [fill=black] (0.9959151997881479,-3.5549930185489833) circle (2.0pt);
 		\draw [fill=black] (-1.508140190803178,0.8706744099242065) circle (2.0pt);
 		\draw [fill=black] (-1.505065505804085,-0.868930288604922) circle (2.0pt);
 		\draw [fill=black] (1.5268383214145342,-0.8809756102856563) circle (2.0pt);
 		\draw [fill=black] (1.5506830978180768,0.893372442721198) circle (2.0pt);
 		\end{scriptsize}
 		\end{tikzpicture}
 	\end{minipage}
 	\begin{minipage}{0.32\textwidth}
 		\definecolor{xdxdff}{rgb}{0.49019607843137253,0.49019607843137253,1.}
 		\definecolor{uuuuuu}{rgb}{0.26666666666666666,0.26666666666666666,0.26666666666666666}
 		\begin{tikzpicture}[line cap=round,line join=round,>=triangle 45,x=1.0cm,y=1.0cm,scale=0.7]
 		\clip(-3.,-2.8) rectangle (3.,2.8);
 		\draw(0.,0.) circle (1.cm);
 		\draw(2.,0.) circle (1.cm);
 		\draw(-2.,0.) circle (1.cm);
 		\draw(-1.,1.7320508075688772) circle (1.0000878721036468cm);
 		\draw(1.,1.7320508075688772) circle (1.0033111550690994cm);
 		\draw(3.,1.7320508075688772) circle (1.001878446163976cm);
 		\draw(-3.,1.7320508075688772) circle (1.000031327963911cm);
 		\draw(-1.,-1.7320508075688772) circle (1.000034097123565cm);
 		\draw(1.,-1.7320508075688772) circle (1.0009433592074812cm);
 		\draw(3.,-1.7320508075688772) circle (1.0034182825738462cm);
 		\draw(-3.,-1.7320508075688772) circle (1.0000040084275728cm);
 		\draw(-2.0064808217537324,3.4641016151377544) circle (1.0053548382577902cm);
 		\draw(0.0032482965229566887,3.4641016151377544) circle (0.9952221677733265cm);
 		\draw(1.997806592754283,3.4641016151377544) circle (0.9990427477938433cm);
 		\draw(-2.0055568446814505,-3.4641016151377544) circle (1.0027583052641529cm);
 		\draw(0.,-3.4641016151377544) circle (1.0000541647246564cm);
 		\draw(1.9835597477812021,-3.4641016151377544) circle (0.9918180278632107cm);
 		\draw [dash pattern=on 2pt off 2pt] (-1.0172414183184653,2.863818722840238) circle (0.6007314551384857cm);
 		\draw [dash pattern=on 2pt off 2pt] (0.004587163006343346,2.2617239954505273) circle (0.5880317567061172cm);
 		\draw [dash pattern=on 2pt off 2pt] (1.0258751724803175,2.8419803490141895) circle (0.590986927432711cm);
 		\draw [dash pattern=on 2pt off 2pt] (-2.083360360941715,2.314481706281469) circle (0.5991720963395554cm);
 		\draw [dash pattern=on 2pt off 2pt] (-2.0085989887447506,1.1486973497848278) circle (0.572499575131604cm);
 		\draw [dash pattern=on 2pt off 2pt] (-1.0108242957494513,0.5772995339779164) circle (0.5774010021723125cm);
 		\draw [dash pattern=on 2pt off 2pt] (0.026340438875311015,1.0922095257628304) circle (0.582186954820034cm);
 		\draw [dash pattern=on 2pt off 2pt] (2.056183608801216,1.1120088181784948) circle (0.550756417372797cm);
 		\draw [dash pattern=on 2pt off 2pt] (2.0494818680297278,2.250483808290168) circle (0.5925285596862527cm);
 		\draw [dash pattern=on 2pt off 2pt] (1.0664719181097535,0.5754354079577545) circle (0.5792619654600992cm);
 		\draw [dash pattern=on 2pt off 2pt] (1.,-0.5980169204045331) circle (0.5980169204045331cm);
 		\draw [dash pattern=on 2pt off 2pt] (1.9829526282810157,-1.226272539038344) circle (0.5720753708510177cm);
 		\draw [dash pattern=on 2pt off 2pt] (2.92950881170672,-0.5735520495312157) circle (0.5778675982856772cm);
 		\draw [dash pattern=on 2pt off 2pt] (3.005468969049834,0.5517947516375288) circle (0.5518218530986149cm);
 		\draw [dash pattern=on 2pt off 2pt] (-0.02164092242999822,-1.1836633216719006) circle (0.5800426403256216cm);
 		\draw [dash pattern=on 2pt off 2pt] (-1.0067426418371948,-0.5773305828705169) circle (0.5773699551730723cm);
 		\draw [dash pattern=on 2pt off 2pt] (-2.0022283022364014,-1.1614122075330537) circle (0.5768158450114591cm);
 		\draw [dash pattern=on 2pt off 2pt] (-0.9941193434915839,-2.9252308879504167) circle (0.6019860401490646cm);
 		\draw [dash pattern=on 2pt off 2pt] (-0.012736500333227648,-2.316538954886902) circle (0.5533563441130747cm);
 		\draw [dash pattern=on 2pt off 2pt] (0.9579392115078053,-2.948353794035274) circle (0.6078267223514151cm);
 		\draw [dash pattern=on 2pt off 2pt] (-3.047513494828275,0.6008797387678011) circle (0.602755333989221cm);
 		\draw [dash pattern=on 2pt off 2pt] (-3.02649170860371,-0.5940161553200489) circle (0.5946065954948334cm);
 		\draw [dash pattern=on 2pt off 2pt] (-1.9573935131823972,-2.289045327839835) circle (0.5524680474590185cm);
 		\draw [dash pattern=on 2pt off 2pt] (2.0067577136567487,-2.3850712751633285) circle (0.5869962826053248cm);
 		\begin{scriptsize}
 		\draw [fill=black] (1.,0.) circle (2.0pt);
 		\draw [fill=black] (-1.,0.) circle (2.0pt);
 		\draw [fill=black] (-3.,0.) circle (2.0pt);
 		\draw [fill=black] (3.,0.) circle (2.0pt);
 		\draw [fill=black] (-0.5080962528209986,0.8613002948271061) circle (2.0pt);
 		\draw [fill=black] (0.5490247793465535,0.8358060729998725) circle (2.0pt);
 		\draw [fill=black] (2.537073188919372,0.8435356481761612) circle (2.0pt);
 		\draw [fill=black] (-2.5048394513469736,0.8632132577548184) circle (2.0pt);
 		\draw [fill=black] (-0.5050484569517069,-0.8630909894852917) circle (2.0pt);
 		\draw [fill=black] (0.47316167858547153,-0.8809756102856531) circle (2.0pt);
 		\draw [fill=black] (2.448489301586677,-0.8937881999457673) circle (2.0pt);
 		\draw [fill=black] (-2.4982651225517767,-0.867024721474805) circle (2.0pt);
 		\draw [fill=black] (-1.5424626932845245,2.5722352982156043) circle (2.0pt);
 		\draw [fill=black] (0.48880317675156626,2.595364818683489) circle (2.0pt);
 		\draw [fill=black] (1.550871543308617,2.5706054072574562) circle (2.0pt);
 		\draw [fill=black] (-0.9940308515927064,3.4641016151377544) circle (2.0pt);
 		\draw [fill=black] (-0.4803706777031921,2.59428562827853) circle (2.0pt);
 		\draw [fill=black] (-0.0016143804893224978,1.6737249412394237) circle (2.0pt);
 		\draw [fill=black] (-2.000028273754884,1.7211327665528937) circle (2.0pt);
 		\draw [fill=black] (2.000718213249537,1.659965222624701) circle (2.0pt);
 		\draw [fill=black] (-2.0000150446346328,-1.7382238063650053) circle (2.0pt);
 		\draw [fill=black] (-4.548532329533561E-4,-1.7633189220657002) circle (2.0pt);
 		\draw [fill=black] (1.9987598396345276,-1.798129480942134) circle (2.0pt);
 		\draw [fill=black] (-1.4996165869778435,-2.598336823516471) circle (2.0pt);
 		\draw [fill=black] (-0.4918688111991666,-2.5933700807641284) circle (2.0pt);
 		\draw [fill=black] (1.4675578962670535,-2.6170803110205325) circle (2.0pt);
 		\draw [fill=black] (0.9992679650524392,3.4323680224197624) circle (2.0pt);
 		\draw [fill=black] (-0.9980613301393708,-3.5272040213007516) circle (2.0pt);
 		\draw [fill=black] (0.4724219809560179,-2.5826676673613367) circle (2.0pt);
 		\draw [fill=black] (0.9959151997881479,-3.5549930185489833) circle (2.0pt);
 		\draw [fill=black] (-1.508140190803178,0.8706744099242065) circle (2.0pt);
 		\draw [fill=black] (-1.505065505804085,-0.868930288604922) circle (2.0pt);
 		\draw [fill=black] (1.5268383214145342,-0.8809756102856563) circle (2.0pt);
 		\draw [fill=black] (1.5506830978180768,0.893372442721198) circle (2.0pt);
 		\draw [fill=black] (-2.438946886302586,-2.5598354537671817) circle (2.0pt);
 		\draw [fill=black] (2.535488052485511,-2.6400396351503894) circle (2.0pt);
 		\draw [fill=black] (-3.54144256513851,-0.8913082991113571) circle (2.0pt);
 		\draw [fill=black] (-2.585121176187902,2.6419619122717233) circle (1.0pt);
 		\draw [fill=black] (2.519075568865809,2.611831079179181) circle (2.0pt);
 		\draw [fill=black] (3.471141005751571,0.8478638513275284) circle (2.0pt);
 		\draw [fill=black] (3.4458067208149203,-0.8331044145733355) circle (2.0pt);
 		\draw [fill=black] (-3.5662768084215806,0.9077976709942375) circle (2.0pt);
 		
 		\end{scriptsize}
 		\end{tikzpicture}
 	\end{minipage}
 	\begin{minipage}{0.32\textwidth}
 		\begin{tikzpicture}[line cap=round,line join=round,>=triangle 45,x=1.0cm,y=1.0cm,scale=0.7]
 		\clip(-3.,-2.8) rectangle (3.,2.8);
 		\draw [dash pattern=on 2pt off 2pt] (0.,0.) circle (1.cm);
 		\draw [dash pattern=on 2pt off 2pt] (2.,0.) circle (1.cm);
 		\draw [dash pattern=on 2pt off 2pt] (-2.,0.) circle (1.cm);
 		\draw [dash pattern=on 2pt off 2pt] (-1.,1.7320508075688772) circle (1.0000878721036468cm);
 		\draw [dash pattern=on 2pt off 2pt] (1.,1.7320508075688772) circle (1.0033111550690994cm);
 		\draw [dash pattern=on 2pt off 2pt] (3.,1.7320508075688772) circle (1.001878446163976cm);
 		\draw [dash pattern=on 2pt off 2pt] (-3.,1.7320508075688772) circle (1.000031327963911cm);
 		\draw [dash pattern=on 2pt off 2pt] (-1.,-1.7320508075688772) circle (1.000034097123565cm);
 		\draw [dash pattern=on 2pt off 2pt] (1.,-1.7320508075688772) circle (1.0009433592074812cm);
 		\draw [dash pattern=on 2pt off 2pt] (3.,-1.7320508075688772) circle (1.0034182825738462cm);
 		\draw [dash pattern=on 2pt off 2pt] (-3.,-1.7320508075688772) circle (1.0000040084275728cm);
 		\draw [dash pattern=on 2pt off 2pt] (-2.0064808217537324,3.4641016151377544) circle (1.0053548382577902cm);
 		\draw [dash pattern=on 2pt off 2pt] (0.0032482965229566887,3.4641016151377544) circle (0.9952221677733265cm);
 		\draw [dash pattern=on 2pt off 2pt] (1.997806592754283,3.4641016151377544) circle (0.9990427477938433cm);
 		\draw [dash pattern=on 2pt off 2pt] (-2.0055568446814505,-3.4641016151377544) circle (1.0027583052641529cm);
 		\draw [dash pattern=on 2pt off 2pt] (0.,-3.4641016151377544) circle (1.0000541647246564cm);
 		\draw [dash pattern=on 2pt off 2pt] (1.9835597477812021,-3.4641016151377544) circle (0.9918180278632107cm);
 		\draw [dash pattern=on 2pt off 2pt] (-1.0172414183184653,2.863818722840238) circle (0.6007314551384857cm);
 		\draw [dash pattern=on 2pt off 2pt] (0.004587163006343346,2.2617239954505273) circle (0.5880317567061172cm);
 		\draw [dash pattern=on 2pt off 2pt] (1.0258751724803175,2.8419803490141895) circle (0.590986927432711cm);
 		\draw [dash pattern=on 2pt off 2pt] (-2.083360360941715,2.314481706281469) circle (0.5991720963395554cm);
 		\draw [dash pattern=on 2pt off 2pt] (-2.0085989887447506,1.1486973497848278) circle (0.572499575131604cm);
 		\draw [dash pattern=on 2pt off 2pt] (-1.0108242957494513,0.5772995339779164) circle (0.5774010021723125cm);
 		\draw [dash pattern=on 2pt off 2pt] (0.026340438875311015,1.0922095257628304) circle (0.582186954820034cm);
 		\draw [dash pattern=on 2pt off 2pt] (2.056183608801216,1.1120088181784948) circle (0.550756417372797cm);
 		\draw [dash pattern=on 2pt off 2pt] (2.0494818680297278,2.250483808290168) circle (0.5925285596862527cm);
 		\draw [dash pattern=on 2pt off 2pt] (1.0664719181097535,0.5754354079577545) circle (0.5792619654600992cm);
 		\draw [dash pattern=on 2pt off 2pt] (1.,-0.5980169204045331) circle (0.5980169204045331cm);
 		\draw [dash pattern=on 2pt off 2pt] (1.9829526282810157,-1.226272539038344) circle (0.5720753708510177cm);
 		\draw [dash pattern=on 2pt off 2pt] (2.9295088117067194,-0.5735520495312157) circle (0.5778675982856772cm);
 		\draw [dash pattern=on 2pt off 2pt] (3.0054689690498346,0.5517947516375288) circle (0.5518218530986152cm);
 		\draw [dash pattern=on 2pt off 2pt] (-0.02164092242999822,-1.1836633216719006) circle (0.5800426403256216cm);
 		\draw [dash pattern=on 2pt off 2pt] (-1.0067426418371948,-0.5773305828705169) circle (0.5773699551730723cm);
 		\draw [dash pattern=on 2pt off 2pt] (-2.0022283022364014,-1.1614122075330537) circle (0.5768158450114591cm);
 		\draw [dash pattern=on 2pt off 2pt] (-0.9941193434915839,-2.9252308879504167) circle (0.6019860401490646cm);
 		\draw [dash pattern=on 2pt off 2pt] (-0.012736500333227648,-2.316538954886902) circle (0.5533563441130747cm);
 		\draw [dash pattern=on 2pt off 2pt] (0.9579392115078053,-2.948353794035274) circle (0.6078267223514151cm);
 		\draw [dash pattern=on 2pt off 2pt] (-3.047513494828275,0.6008797387678011) circle (0.602755333989221cm);
 		\draw [dash pattern=on 2pt off 2pt] (-3.02649170860371,-0.5940161553200489) circle (0.5946065954948334cm);
 		\draw [dash pattern=on 2pt off 2pt] (-1.9573935131823972,-2.289045327839835) circle (0.5524680474590185cm);
 		\draw [dash pattern=on 2pt off 2pt] (2.0067577136567487,-2.3850712751633285) circle (0.5869962826053248cm);
 		\draw (-1.5424626932845245,2.5722352982156043)-- (-2.000028273754884,1.7211327665528937);
 		\draw (-1.508140190803178,0.8706744099242065)-- (-2.000028273754884,1.7211327665528937);
 		\draw (-2.000028273754884,1.7211327665528937)-- (-2.5048394513469736,0.8632132577548184);
 		\draw (-2.5048394513469736,0.8632132577548184)-- (-1.508140190803178,0.8706744099242065);
 		\draw (-1.508140190803178,0.8706744099242065)-- (-0.5080962528209986,0.8613002948271061);
 		\draw (-0.5080962528209986,0.8613002948271061)-- (-0.0016143804893224978,1.6737249412394237);
 		\draw (-0.0016143804893224978,1.6737249412394237)-- (-0.4803706777031921,2.59428562827853);
 		\draw (-0.4803706777031921,2.59428562827853)-- (-1.5424626932845245,2.5722352982156043);
 		\draw (-1.5424626932845245,2.5722352982156043)-- (-0.9940308515927064,3.4641016151377544);
 		\draw (-0.9940308515927064,3.4641016151377544)-- (-0.4803706777031921,2.59428562827853);
 		\draw (-0.4803706777031921,2.59428562827853)-- (0.48880317675156626,2.595364818683489);
 		\draw (0.48880317675156626,2.595364818683489)-- (-0.0016143804893224978,1.6737249412394237);
 		\draw (-0.0016143804893224978,1.6737249412394237)-- (0.5490247793465535,0.8358060729998725);
 		\draw (0.5490247793465535,0.8358060729998725)-- (1.5506830978180768,0.893372442721198);
 		\draw (1.5506830978180768,0.893372442721198)-- (2.000718213249537,1.659965222624701);
 		\draw (2.000718213249537,1.659965222624701)-- (1.550871543308617,2.5706054072574562);
 		\draw (1.550871543308617,2.5706054072574562)-- (0.9992679650524392,3.4323680224197624);
 		\draw (0.9992679650524392,3.4323680224197624)-- (0.48880317675156626,2.595364818683489);
 		\draw (0.48880317675156626,2.595364818683489)-- (1.550871543308617,2.5706054072574562);
 		\draw (1.5506830978180768,0.893372442721198)-- (2.537073188919372,0.8435356481761612);
 		\draw (2.537073188919372,0.8435356481761612)-- (2.000718213249537,1.659965222624701);
 		\draw (2.537073188919372,0.8435356481761612)-- (3.,0.);
 		\draw (3.,0.)-- (2.448489301586677,-0.8937881999457673);
 		\draw (2.448489301586677,-0.8937881999457673)-- (1.5268383214145342,-0.8809756102856563);
 		\draw (1.5268383214145342,-0.8809756102856563)-- (1.,0.);
 		\draw (1.,0.)-- (1.5506830978180768,0.893372442721198);
 		\draw (0.5490247793465535,0.8358060729998725)-- (1.,0.);
 		\draw (1.,0.)-- (0.47316167858547153,-0.8809756102856531);
 		\draw (0.47316167858547153,-0.8809756102856531)-- (-0.5050484569517069,-0.8630909894852917);
 		\draw (-0.5050484569517069,-0.8630909894852917)-- (-1.,0.);
 		\draw (-1.,0.)-- (-0.5080962528209986,0.8613002948271061);
 		\draw (-0.5080962528209986,0.8613002948271061)-- (0.5490247793465535,0.8358060729998725);
 		\draw (-1.508140190803178,0.8706744099242065)-- (-1.,0.);
 		\draw (-1.,0.)-- (-1.505065505804085,-0.868930288604922);
 		\draw (-1.505065505804085,-0.868930288604922)-- (-2.4982651225517767,-0.867024721474805);
 		\draw (-2.4982651225517767,-0.867024721474805)-- (-3.,0.);
 		\draw (-3.,0.)-- (-2.5048394513469736,0.8632132577548184);
 		\draw (-2.5048394513469736,0.8632132577548184)-- (-3.5662768084215806,0.9077976709942375);
 		\draw (-2.000028273754884,1.7211327665528937)-- (-2.585121176187902,2.6419619122717233);
 		\draw (-2.585121176187902,2.6419619122717233)-- (-1.5424626932845245,2.5722352982156043);
 		\draw (-3.,0.)-- (-3.54144256513851,-0.8913082991113571);
 		\draw (-3.54144256513851,-0.8913082991113571)-- (-2.4982651225517767,-0.867024721474805);
 		\draw (-2.4982651225517767,-0.867024721474805)-- (-2.0000150446346328,-1.7382238063650053);
 		\draw (-2.0000150446346328,-1.7382238063650053)-- (-1.505065505804085,-0.868930288604922);
 		\draw (-1.505065505804085,-0.868930288604922)-- (-0.5050484569517069,-0.8630909894852917);
 		\draw (-0.5050484569517069,-0.8630909894852917)-- (-4.548532329533561E-4,-1.7633189220657002);
 		\draw (-4.548532329533561E-4,-1.7633189220657002)-- (0.47316167858547153,-0.8809756102856531);
 		\draw (0.47316167858547153,-0.8809756102856531)-- (1.5268383214145342,-0.8809756102856563);
 		\draw (1.5268383214145342,-0.8809756102856563)-- (1.9987598396345276,-1.798129480942134);
 		\draw (1.9987598396345276,-1.798129480942134)-- (2.448489301586677,-0.8937881999457673);
 		\draw (2.448489301586677,-0.8937881999457673)-- (3.4458067208149203,-0.8331044145733354);
 		\draw (1.9987598396345276,-1.798129480942134)-- (2.535488052485511,-2.6400396351503894);
 		\draw (2.535488052485511,-2.6400396351503894)-- (1.4675578962670535,-2.6170803110205325);
 		\draw (1.4675578962670535,-2.6170803110205325)-- (1.9987598396345276,-1.798129480942134);
 		\draw (-4.548532329533561E-4,-1.7633189220657002)-- (0.4724219809560179,-2.5826676673613367);
 		\draw (0.4724219809560179,-2.5826676673613367)-- (1.4675578962670535,-2.6170803110205325);
 		\draw (1.4675578962670535,-2.6170803110205325)-- (0.9959151997881479,-3.5549930185489833);
 		\draw (0.9959151997881479,-3.5549930185489833)-- (0.4724219809560179,-2.5826676673613367);
 		\draw (0.4724219809560179,-2.5826676673613367)-- (-0.4918688111991666,-2.5933700807641284);
 		\draw (-0.4918688111991666,-2.5933700807641284)-- (-4.548532329533561E-4,-1.7633189220657002);
 		\draw (-2.0000150446346328,-1.7382238063650053)-- (-1.4996165869778435,-2.598336823516471);
 		\draw (-1.4996165869778435,-2.598336823516471)-- (-0.4918688111991666,-2.5933700807641284);
 		\draw (-0.4918688111991666,-2.5933700807641284)-- (-0.9980613301393708,-3.5272040213007516);
 		\draw (-0.9980613301393708,-3.5272040213007516)-- (-1.4996165869778435,-2.598336823516471);
 		\draw (-1.4996165869778435,-2.598336823516471)-- (-2.438946886302586,-2.5598354537671817);
 		\draw (-2.438946886302586,-2.5598354537671817)-- (-2.0000150446346328,-1.7382238063650053);
 		\draw (-3.5662768084215806,0.9077976709942375)-- (-3.,0.);
 		\draw (3.,0.)-- (3.471141005751571,0.8478638513275283);
 		\draw (3.471141005751571,0.8478638513275283)-- (2.537073188919372,0.8435356481761612);
 		\draw (1.550871543308617,2.5706054072574562)-- (2.519075568865809,2.611831079179181);
 		\draw (2.519075568865809,2.611831079179181)-- (2.000718213249537,1.659965222624701);
 		\draw (3.,0.)-- (3.4458067208149203,-0.8331044145733354);
 		\begin{scriptsize}
 		\draw [fill=black] (1.,0.) circle (2.0pt);
 		\draw [fill=black] (-1.,0.) circle (2.0pt);
 		\draw [fill=black] (-3.,0.) circle (2.0pt);
 		\draw [fill=black] (3.,0.) circle (2.0pt);
 		\draw [fill=black] (-0.5080962528209986,0.8613002948271061) circle (2.0pt);
 		\draw [fill=black] (0.5490247793465535,0.8358060729998725) circle (2.0pt);
 		\draw [fill=black] (2.537073188919372,0.8435356481761612) circle (2.0pt);
 		\draw [fill=black] (-2.5048394513469736,0.8632132577548184) circle (2.0pt);
 		\draw [fill=black] (-0.5050484569517069,-0.8630909894852917) circle (2.0pt);
 		\draw [fill=black] (0.47316167858547153,-0.8809756102856531) circle (2.0pt);
 		\draw [fill=black] (2.448489301586677,-0.8937881999457673) circle (2.0pt);
 		\draw [fill=black] (-2.4982651225517767,-0.867024721474805) circle (2.0pt);
 		\draw [fill=black] (-1.5424626932845245,2.5722352982156043) circle (2.0pt);
 		\draw [fill=black] (0.48880317675156626,2.595364818683489) circle (2.0pt);
 		\draw [fill=black] (1.550871543308617,2.5706054072574562) circle (2.0pt);
 		\draw [fill=black] (-0.9940308515927064,3.4641016151377544) circle (2.0pt);
 		\draw [fill=black] (-0.4803706777031921,2.59428562827853) circle (2.0pt);
 		\draw [fill=black] (-0.0016143804893224978,1.6737249412394237) circle (2.0pt);
 		\draw [fill=black] (-2.000028273754884,1.7211327665528937) circle (2.0pt);
 		\draw [fill=black] (2.000718213249537,1.659965222624701) circle (2.0pt);
 		\draw [fill=black] (-2.0000150446346328,-1.7382238063650053) circle (2.0pt);
 		\draw [fill=black] (-4.548532329533561E-4,-1.7633189220657002) circle (2.0pt);
 		\draw [fill=black] (1.9987598396345276,-1.798129480942134) circle (2.0pt);
 		\draw [fill=black] (-1.4996165869778435,-2.598336823516471) circle (2.0pt);
 		\draw [fill=black] (-0.4918688111991666,-2.5933700807641284) circle (2.0pt);
 		\draw [fill=black] (1.4675578962670535,-2.6170803110205325) circle (2.0pt);
 		\draw [fill=black] (0.9992679650524392,3.4323680224197624) circle (2.0pt);
 		\draw [fill=black] (-0.9980613301393708,-3.5272040213007516) circle (2.0pt);
 		\draw [fill=black] (0.4724219809560179,-2.5826676673613367) circle (2.0pt);
 		\draw [fill=black] (0.9959151997881479,-3.5549930185489833) circle (2.0pt);
 		\draw [fill=black] (-1.508140190803178,0.8706744099242065) circle (2.0pt);
 		\draw [fill=black] (-1.505065505804085,-0.868930288604922) circle (2.0pt);
 		\draw [fill=black] (1.5268383214145342,-0.8809756102856563) circle (2.0pt);
 		\draw [fill=black] (1.5506830978180768,0.893372442721198) circle (2.0pt);
 		\draw [fill=black] (-2.585121176187902,2.6419619122717233) circle (2.0pt);
 		\draw [fill=black] (2.519075568865809,2.611831079179181) circle (2.0pt);
 		\draw [fill=black] (3.471141005751571,0.8478638513275283) circle (2.0pt);
 		\draw [fill=black] (3.4458067208149203,-0.8331044145733354) circle (2.0pt);
 		\draw [fill=black] (2.535488052485511,-2.6400396351503894) circle (2.0pt);
 		\draw [fill=black] (-2.438946886302586,-2.5598354537671817) circle (2.0pt);
 		\draw [fill=black] (-3.54144256513851,-0.8913082991113571) circle (2.0pt);
 		\draw [fill=black] (-3.5662768084215806,0.9077976709942375) circle (2.0pt);
 		\end{scriptsize}
 		\end{tikzpicture}
 	\end{minipage}
 	\caption{A circle packing with the combinatorics of a triangle mesh (left) induces a dual circle packing (center). The two together form an orthogonal circle pattern. The tangency points become the vertices of the medial graph of the given triangle mesh. On the right, the medial graph consists of hexagonal faces and triangular faces.}
 	\label{fig:circlepacking}
 \end{figure}
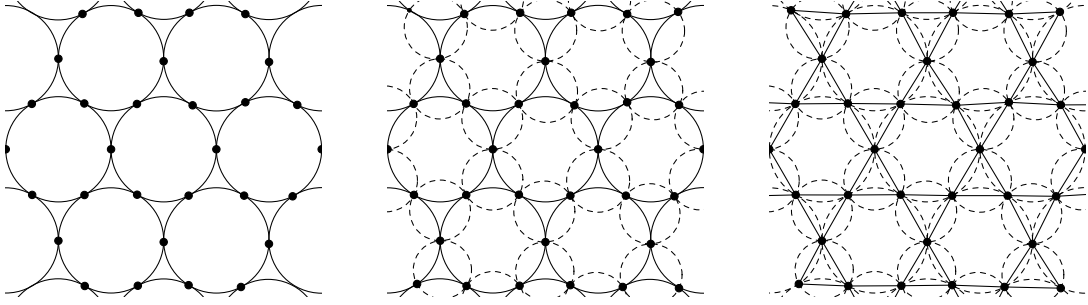
 
 The tangency points of a circle packing form a realization of the \textbf{medial graph} (Fig. \ref{fig:circlepacking} right) with vertices as the intersection points of the induced orthogonal circle pattern. Combinatorially, the medial graph $MG$ has a vertex for every edge of $G$. Two vertices of $MG$ are connected by an edge if their corresponding edges of $G$ occur consecutively in a face of $G$. Each face of $MG$ corresponds either to a face of $G$ or an interior vertex of $G$. We denote $TMG$ as a triangulation of the medial graph $MG$. 
 
 \textbf{M\"{o}bius transformations} in $\mathbb{C}\cup \{\infty\}$ are generated by Euclidean motions and inversions. They are conformal and map circles to circles. Trivial deformations of circle packings are generated by M\"{o}bius transformations. Since we are interested in continuous deformations of a circle packing, without further notice we focus on orientation-preserving M\"{o}bius transformations, which are in the form $z \mapsto \frac{a z + b}{c z +d}$ for some $a,b,c,d \in \mathbb{C}$ such that $ad - bc =1$.
 
In \cite{Lam2015a}, \textbf{discrete holomorphic quadratic differentials} were introduced to parametrize infinitesimal deformations of circle patterns preserving the intersection angles (see Proposition \ref{prop:infmobgen}).
 
 \begin{definition} [Lam-Pinkall \cite{Lam2015a}] \label{def:hologen}
 	Let $z:V \to \mathbb{C}$ be a realization of a triangle mesh with $z_i \neq z_j$ for all $\{ij\} \in E$. A function  $q:E_{int} \to \mathbb{R}$ is called a holomorphic quadratic differential of general type if it satisfies for every interior vertex $i \in V_{int}$
 	\begin{align*} \label{eq:qgeneral}
 	\sum_j q_{ij}&=0 \\
 	\sum_j \frac{q_{ij}}{z_j -z_i} &=0
 	\end{align*}
 	where $q_{ij}= q_{ji}$ and the sum is taken over all the edges connecting to $i$.
 \end{definition}
 
 In the case of a circle packing with the combinatorics of a triangle mesh $G$, an orthogonal circle pattern is induced. We will consider  tangency points of the circles as the realization $z:V(TMG)\to \mathbb{C}$ and holomorphic quadratic differentials $q:E(TMG) \to \mathbb{R}$ defined on a triangulation of the medial graph.  
 
 \textbf{Discrete minimal surfaces of general type} in Definition \ref{def:genmin} can be constructed from holomorphic quadratic differentials via a Weierstrass representation formula.
 
 \begin{proposition}[Lam \cite{Lam2016}] \label{thm:Weierstrass}
 	Suppose $z:V \to \mathbb{C}$ is a realization of a triangular mesh and $q:E_{int} \to \mathbb{R}$ is a holomorphic quadratic differential. Then there exists $\hat{\mathcal{F}}: F \to \mathbb{C}^3$ such that for every edge $\{ij\} \in E_{int}$
 	\begin{equation} \label{eq:eta}
 	\hat{\mathcal{F}}_{\text{left}(e_{ij})}- \hat{\mathcal{F}}_{\text{right}(e_{ij})} =\frac{q_{ij}}{z_j - z_i} \left( \begin{array}{c}
 	1-z_i z_j \\ \mathbf{i}(1+z_i z_j) \\ z_i + z_j
 	\end{array}\right)
 	\end{equation}
 	where $\text{left}(e_{ij}), \text{right}(e_{ij})$ denote the left and the right face of $e_{ij}$. We consider  $N:V \to \mathbb{S}^2$ the stereographic projection of $z$
 	\begin{equation*}
 	N := \frac{1}{1+|z|^2} \left(\begin{array}{c}2\Re z \\ 2\Im z \\ |z|^2-1\end{array}\right)		
 	\end{equation*}
 	and assume $N_i \neq -N_j$ for $\{ij\} \in E$. Then we have
 	\begin{enumerate}
 			\item $\Im(\hat{\mathcal{F}}):F \to \mathbb{R}^3$ is a discrete minimal surface of general type.
 		\item $\Re(\hat{\mathcal{F}}):F \to \mathbb{R}^3$ is a reciprocal parallel mesh of $N$.
 	\end{enumerate}
 	
 	The converse also holds: given a reciprocal parallel mesh of $N$ or a polyhedral surface with $H \equiv 0$ and face normal $N$, there exists a holomorphic quadratic differential $q$ on $z$ satisfying Eq. \eqref{eq:eta}.
 \end{proposition}
 
 $\Re(\hat{\mathcal{F}})$ is called a \emph{reciprocal parallel mesh} of $N$ since it is defined on the dual graph such that each dual edge is parallel to its primal edge. The analogue of considering minimal surfaces as reciprocal parallel meshes in the smooth theory is known to Liebmann and Sabitov \cite[p. 199]{Sabitov1992}.
 
 The two surfaces $\Re(\hat{\mathcal{F}})$ and $\Im(\hat{\mathcal{F}})$ form a conjugate pair of minimal surfaces, though the realization $\Re(\hat{\mathcal{F}})$ is generally not a polyhedral surface.  

\section{Parametrizations of circle packings}  \label{sec:param}

 We consider two equivalent parametrizations of circle packings -- \emph{cross ratios as circle packing} and \emph{cross ratios as circle patterns} which is key to the unification of discrete minimal surfaces.

\subsection{Cross ratios as circle packings} He and Schramm \cite{He1998} proposed to describe circle packings in terms of cross ratios since they are invariant under M\"{o}bius transformations. Suppose $\{\mathfrak{uv}\} \in E(G) $ is an interior edge shared by two neighboring faces $\{\mathfrak{uvs}\},\{\mathfrak{vut}\}$. We denote $z_i$ the point of tangency of the circles $C_\mathfrak{u}$ and $C_\mathfrak{v}$. The two circles $C_\mathfrak{s}, C_{\mathfrak{t}}$ touch respectively $C_\mathfrak{u}$ and $C_\mathfrak{v}$ at $z_j,z_k,z_m,z_n$ as in Fig. \ref{fig:orientation} (left). We associate the edge $\{\mathfrak{uv}\}$ with the cross ratio 
\[
\cratio^{\dagger}_{\mathfrak{uv}} := \cratio(z_i,z_j,z_k,z_n)=  \cratio(z_i,z_m,z_n,z_j).
\] 
Mapping $z_i$ to infinity by inversion, the image of $z_j,z_k,z_m,z_n$ form a rectangle. The cross ratio $\cratio^{\dagger}$ is then the aspect ratio of the rectangle and hence is purely imaginary. It defines a purely imaginary function $\cratio^\dagger : E_{int}(G) \to \mathbf{i} \mathbb{R}$. The following is straightforward \cite{He1998, Kojima2003}.

\begin{proposition} \label{thm:crpacking}
	Suppose $C$ and $\tilde{C}$ are two circle packings of the combinatorics of a triangle mesh in the plane. Then they differ by a M\"{o}bius transformation if and only if 
	\[
	\cratio^{\dagger} \equiv \tilde{\cratio}^{\dagger}.
	\]
\end{proposition} 

\subsection{Cross ratios as circle patterns} 
Every circle packing is associated with a dual circle packing and the two together form an orthogonal circle pattern (Fig. \ref{fig:circlepacking} center). The points of tangency in the circle packing induce a realization of the medial graph of the triangle mesh (Fig. \ref{fig:circlepacking} right). The vertices are then regarded as the intersection points of the induced circle pattern.

The faces in the media graph $MG$ formed by the tangency points of the circle packing are bipartite. One kind of faces correspond to the circles in the packing and the other kind of faces are triangular which are associated to the interstices (Figure \ref{fig:circlepacking} right). We triangulate the media graph $MG$ without adding new vertices and obtain a triangulation $TMG$. Every unoriented edge $\{ij\} \in E(TMG)$ is associated with a complex number
\begin{align}\label{eq:cratiocpattern}
\cratio_{ij} := \cratio(z_j,z_k,z_i,z_l) = \frac{(z_j-z_k)(z_i-z_l)}{(z_k-z_i)(z_l-z_j)} = \cratio_{ji}
\end{align}
where $\{ijk\}$ and $\{jil\}$ are two neighboring triangles (Fig. \ref{fig:orientation} right). In this way a function $\cratio: E_{int}(TMG) \to \mathbb{C}$ is induced. The argument of the cross ratio indicates the intersection angle of the neighboring circumscribed circles, which implies $\Arg \cratio =\pi/2 \, \text{or} \, \pi$ for the induced orthogonal circle pattern. The following observation is immediate:

\begin{proposition} \label{thm:crpattern}
	Suppose $TMG$ a triangulation of the medial graph and $z:V(TMG)\to \mathbb{C}$ is the tangency points of a circle packing $C$. Then $\tilde{z}:V(TMG)\to \mathbb{C}$ is the tangency points of another circle packing $\tilde{C}$ if and only if
	\[
	\Arg \cratio \equiv \Arg \tilde{\cratio}.
	\]
	Furthermore, the circle packings differ by a M\"{o}bius transformation if and only if
	\[
	\cratio \equiv \tilde{\cratio}.
	\]
\end{proposition}

In particular it implies that the cross ratios of two circle packings only differ in the magnitude  $|\cratio|$.

\section{Infinitesimal deformations of circle packings} \label{sec:verroation}

We consider the change in cross ratios under  infinitesimal deformations, which motivates the definition of holomorphic quadratic differentials. 

\subsection{Holomorphic quadratic differentials of general type} \label{sec:hqdg}

In this subsection, we review a result in \cite{Lam2015a} that studied the change in cross ratios $\cratio$ under infinitesimal deformations of circle patterns that preserve the intersection angles. Recall that there exists a unique M\"{o}bius transformation $\Gamma \in SL(2, \mathbb{C})/\{Id,-Id\}$ (i.e. a fractional linear transformation) that maps any three distinct points $z_i$ in the plane to any other three distinct points $\tilde{z}_i$:
\[
\Gamma \left(\begin{array}{c}
z_i \\ 1
\end{array}\right) = \beta_{i} \left(\begin{array}{c}
\tilde{z}_i \\ 1
\end{array}\right) \quad \text{for some } \beta_i \in \mathbb{C} 
\]

On the infinitesimal level, one can also deduce that any infinitesimal deformation of three distinct points in the plane is induced by a unique infinitesimal M\"{o}bius transformation $\Phi \in sl(2,\mathbb{C}) $, that is a 2 by 2 matrix with trace zero.
\begin{lemma}\label{lem:infmob}
	Suppose $z_1,z_2,z_3 \in \mathbb{C}$ are three distinct points and $\dot{z}_1, \dot{z}_2, \dot{z}_3 \in \mathbb{C}$ are three arbitrary vectors. Then there exists a unique $\Phi \in sl(2,\mathbb{C})$  such that for $i=1,2,3$
	\[
	\Phi \left(\begin{array}{c}
	z_i \\ 1
	\end{array}\right) = \mu_{i} \left(\begin{array}{c}
	z_i \\ 1
	\end{array}\right) + \left(\begin{array}{c}
	\dot{z}_i \\ 0
	\end{array}\right) 
	\]
	for some $\mu_{i} \in \mathbb{C}$. Equivalently,  $\Phi \in sl(2,\mathbb{C})$ is uniquely determined by
	\[
	\dot{z}_i = \det\left(	\Phi \left(\begin{array}{c}
	z_i \\ 1
	\end{array}\right), \left(\begin{array}{c}
	z_i \\ 1
	\end{array}\right) \right) \quad \text{for } i=1,2,3
	\]
\end{lemma}

Suppose $z:V(MG)\to \mathbb{C}$ denotes the tangency points of a circle packing and $\dot{z}:V(MG) \to \mathbb{C}$ the change in the tangency points induced from an infinitesimal deformation of a circle packing. We take a triangulation $TMG$ of the medial graph and have $V(TMG)=V(MG)$. Then there exists a unique $\hat{\Phi}: F(TMG) \to sl(2,C)$ such that for every face $\{ijk\} \in F(TMG)$  the matrix $\hat{\Phi}_{ijk}$ relates the vertices $z_{i},z_{j}, z_{k}$ to their changes $\dot{z}_i,\dot{z}_j,\dot{z}_k$ via Lemma \ref{lem:infmob}. For an edge $\{ij\} \in E_{int}(TMG)$ shared by $\{ijk\}$ and $\{jil\}$ (see Figure \ref{fig:orientation} right), the matrix $\hat{\Phi}_{ijk} - \hat{\Phi}_{jil}$ thus has eigenvectors $\left(\begin{array}{c}
z_i \\ 1
\end{array}\right)$ and  $\left(\begin{array}{c}
z_j \\ 1
\end{array}\right)$. Since the intersection angle of the circumscribed circles is preserved, the matrix has real eigenvalues $q$ and $ -q$ . 

\begin{proposition}[Lam-Pinkall \cite{Lam2015a}] \label{prop:infmobgen}
	Suppose $z:V(MG)=V(TMG) \to \mathbb{C}$ is the tangency points of a circle packing with the combinatorics of a triangulated disk $G$. 
  Then a function $\hat{\Phi}: F(TMG) \to sl(2,C)$ represents an infinitesimal deformation on the induced orthogonal circle pattern that preserves intersection angles if and only if there exists $q:E_{int}(TMG) \to \mathbb{R}$ satisfying for $\{ij\} \in E_{int}(TMG)$
	\[
	\hat{\Phi}_{ijk} - \hat{\Phi}_{jil} = \frac{q_{ij}}{z_j - z_i} \left( \begin{array}{cc}
	\frac{z_i + z_j}{2} & -z_i z_j \\
	1  & \frac{z_i + z_j}{2} 
	\end{array} \right).
	\]
	In particular, $q$ is a holomorphic quadratic differential of general type and $q = \dot{\cratio}/\cratio$. The deformation is induced by a global M\"{o}bius transformation if and only if $q\equiv 0$.
\end{proposition}

\subsection{Holomorphic quadratic differentials of Koebe type} In this subsection, we consider the change in cross ratios $\cratio^{\dagger}$ under infinitesimal deformations of circle packings.

\begin{lemma} Using the notation of Figure \ref{fig:orientation} left, we define a function $\omega:\vec{E}_{int} \to \mathbb{C}$ on oriented interior edges
	\[
	\omega(e_{\mathfrak{uv}})  := \frac{(z_{k}-z_{i})(z_{m}-z_{i})}{z_{m} - z_{k}} = \frac{(z_{j}-z_{i})(z_{n}-z_{i})}{z_{n} - z_{j}} = -\omega(e_{\mathfrak{vu}})
	\]
	which is a 1-form with value $\omega(e_{\mathfrak{uv}})$ tangent to $C_{\mathfrak{u}}, C_{\mathfrak{v}}$ at their tangency point $z_{i}$.
\end{lemma}
\begin{proof}
	Mapping $z_{i}$ to infinity under inversion, the image of the four neighboring points form a rectangle. Following its parallelogram property, we have
	\[
	\frac{1}{z_{k}-z_{i}} - \frac{1}{z_{m}-z_{i}} = \frac{1}{z_{j}-z_{i}} - \frac{1}{z_{n}-z_{i}}
	\]
	and hence $\omega$ is well defined. We show that it is a vector tangent to the neighboring circles at the point of tangency. Applying a M\"{o}bius transformation $T(z)=\frac{az+b}{cz+d}$ with $ad-bc=1$, we have another circle packing $\tilde{C}$ with tangency points $\tilde{z} = T\circ z$ which defines $\tilde{\omega}$. One can show
	\[
	d T_{z} (\omega(e_{\mathfrak{uv}})) = \frac{\omega(e_{\mathfrak{uv}})}{(c z_i +d)^2} = \tilde{\omega}(e_{\mathfrak{uv}}).
	\]
	In particular, we take $T$ the M\"{o}bius transformation that sends $z_i,z_k,z_m$ to $1,\mathbf{i}=\sqrt{-1},-\mathbf{i}$ and then $\tilde{\omega}(e_{\mathfrak{uv}}) =\mathbf{i}$ is tangent to the unit circle at the image of $z_i$. Thus, the above formula implies $\omega(e_{\mathfrak{uv}})$ is tangent to $C_{\mathfrak{u}}$ and $C_{\mathfrak{v}}$ at $z_i$.
\end{proof}

\begin{definition} \label{thm:infcrpattern}
	Suppose $z:V(TMG)= V(MG) \to \mathbb{C}$ denotes the tangency points of a circle packing $C$. A function $\lambda:E_{int}(G) \to \mathbb{R}$ is called a holomorphic quadratic differential of Koebe type if for every interior vertex $\mathfrak{u} \in V_{int}(G)$
	\begin{equation*} 
	\begin{aligned}
	\sum_\mathfrak{v}  \frac{\lambda_{\mathfrak{uv}}}{\omega(e_{\mathfrak{uv}})} &=0 \\
	\sum_\mathfrak{v}  \frac{\lambda_{\mathfrak{uv}}}{\omega(e_{\mathfrak{uv}})} z_{i} &=0 \\
	\sum_\mathfrak{v}  \frac{\lambda_{\mathfrak{uv}}}{\omega(e_{\mathfrak{uv}})} z_{i}^2 &= 0
	\end{aligned}
	\end{equation*}
	where $\omega(e_{\mathfrak{uv}}) = \frac{(z_{k}-z_{i})(z_{m}-z_{i})}{z_{m} - z_{k}}$. 
\end{definition}

\begin{remark}
	The system of equations are redundant and can be reduced to three real equations per interior vertex.
\end{remark}

 Considering an infinitesimal deformation of a circle packing that preserve the tangency pattern, we denote $\dot{z}:E(G)=V(MG) \to \mathbb{C}$ the infinitesimal change in the tangency points. Lemma \ref{lem:infmob} (see Fig. \ref{fig:orientation} left) implies that there exists a unique $\Phi:F(G) \to sl(2,\mathbb{C})$ such that for every face $\{\mathfrak{uvs}\}$ the matrix $\Phi_{\mathfrak{uvs}}$ relates the tangency points $z_{i}\in C_{\mathfrak{u}} \cap C_{\mathfrak{v}}$, $z_{j}\in C_{\mathfrak{v}} \cap C_{\mathfrak{s}}$, $z_{k}\in C_{\mathfrak{s}} \cap C_{\mathfrak{u}}$ to their changes $\dot{z}_i,\dot{z}_j,\dot{z}_k$. We suppose edge $\{\mathfrak{uv}\}$ is shared by two faces $\{\mathfrak{uvs}\},\{\mathfrak{vut}\}$. Then the matrix  $\Phi_{\mathfrak{uvs}} - \Phi_{\mathfrak{vut}}$ has eigenvector $\left(\begin{array}{c}
z_i \\ 1
\end{array}\right)$. In fact, the difference of the matrices is a linearization of parabolic M\"{o}bius transformations with $z_i$ fixed. Mapping $z_i$ to infinity, the image of $C_{\mathfrak{u}}$ and $C_{\mathfrak{v}}$ becomes two parallel lines and the difference of the matrices turns into an infinitesimal translation in the direction of the parallel lines. A detailed argument is carried out by Orick \cite[Proposition 2.9.1]{Orick2010}, whose linearization leads to quadratic differentials of Koebe type: 

\begin{proposition} \label{prop:infmobkoebe}
	Suppose $z:V(MG) \to \mathbb{C}$ denotes the tangency points of a circle packing with the combinatorics of a triangulated disk $G$. Then a function $\Phi: F(G) \to sl(2,C)$ represents an infinitesimal deformation of the circle packing if and only if there exists $\lambda:E_{int}(G) \to \mathbb{R}$ satisfying for $\{\mathfrak{uv}\} \in E_{int}(G)$
	\begin{equation}\label{eq:orick}
		\Phi_{\mathfrak{uvs}} - \Phi_{\mathfrak{vut}} = \frac{\lambda_{\mathfrak{uv}}}{\omega(e_{\mathfrak{uv}})}\left(\begin{array}{cc}
		z_{i} & -z^2_{i} \\1 & -z_{i}
		\end{array} \right).
	\end{equation}
	In particular, $\lambda$ is a holomorphic quadratic differential of Koebe type and $\lambda = \dot{\cratio}^{\dagger}/\cratio^{\dagger}$. The deformation is induced by a global M\"{o}bius transformation if and only if $\lambda \equiv 0$.
\end{proposition}
\begin{proof}
	Suppose $\Phi$ is given by an infinitesimal deformation of the circle packing, then the form \eqref{eq:orick} follows from Orick \cite[Proposition 2.9.1]{Orick2010}. For any interior vertex $\mathfrak{u} \in V_{int}(G)$, we have
	\[
	\sum_{\mathfrak{v}}\frac{\lambda_{\mathfrak{uv}}}{\omega(e_{\mathfrak{uv}})}\left(\begin{array}{cc}
		z_{i} & -z^2_{i} \\1 & -z_{i}
		\end{array} \right) =\sum_{\mathfrak{v}} (\Phi_{\mathfrak{uvs}} - \Phi_{\mathfrak{vut}})=0 
	\]
	and hence $\lambda$ is a  quadratic differential of Koebe type. Conversely, if $\lambda$ is a quadratic differential of Koebe type, the right hand side of \eqref{eq:orick} defines a closed 1-form on the dual graph. By integration, we obtain $\Phi: F(G) \to sl(2,C)$ which is unique up to a constant. In the notation of Figure \ref{fig:orientation}, we can further define
	\[
	\dot{z}_i :=  \det\left(	\Phi_{\mathfrak{uvs}} \left(\begin{array}{c}
	z_i \\ 1
	\end{array}\right), \left(\begin{array}{c}
	z_i \\ 1
	\end{array}\right) \right) =  \det\left(	\Phi_{\mathfrak{vut}} \left(\begin{array}{c}
	z_i \\ 1
	\end{array}\right), \left(\begin{array}{c}
	z_i \\ 1
	\end{array}\right) \right)
	\]
	since $ \left(\begin{array}{c}
	z_i \\ 1
	\end{array}\right)$ is an eigenvector of $	\Phi_{\mathfrak{uvs}} - 	\Phi_{\mathfrak{vut}}$. It can be verified that $\dot{z}:V(MG) \to \mathbb{C}$ is the change in the tangency points under an infinitesimal deformation of the circle packing. And $\dot{z}$ is unique up to a global infinitesimal M\"{o}bius transformation, which depends on the additive constant from $\Phi$.
	
	It remains to prove that the change in cross ratios is described by $\lambda = \dot{\cratio}^{\dagger}/\cratio^{\dagger}$. Recall that $\Phi_{\mathfrak{vut}}$ maps $z_i,z_m,z_n$ to $\dot{z}_i,\dot{z}_m,\dot{z}_n$ while $\Phi_{\mathfrak{uvs}}$ maps $z_i,z_j,z_k$ to $\dot{z}_i,\dot{z}_j,\dot{z}_k$. We denote $\dot{z}'_m$ the image of $z_m$ under $\Phi_{\mathfrak{uvs}}$. Then
	\begin{align} \label{eq:dercross}
	\begin{split}
	\frac{\dot{\cratio}_{\mathfrak{uv}}^{\dagger}}{\cratio^{\dagger}_{\mathfrak{uv}}} =& \frac{\dot{z}_i -\dot{z}_j}{z_i - z_j} - \frac{\dot{z}_j -\dot{z}_k}{z_j - z_k} + \frac{\dot{z}_k -\dot{z}_m}{z_k - z_m} -\frac{\dot{z}_m -\dot{z}_i}{z_m - z_i} \\
	=& (\frac{\dot{z}_i -\dot{z}_j}{z_i - z_j} - \frac{\dot{z}_j -\dot{z}_k}{z_j - z_k} + \frac{\dot{z}_k -\dot{z}'_m}{z_k - z_m} -\frac{\dot{z}'_m -\dot{z}_i}{z_m - z_i}) + \frac{\dot{z}'_m -\dot{z}_m}{z_k - z_m} -\frac{\dot{z}_m -\dot{z}'_m}{z_m - z_i}  \\
	=& 0 + (\dot{z}'_m -\dot{z}_m) \frac{z_k -z_i}{(z_k - z_m)(z_m-z_i)}
	\end{split}
	\end{align}
	where we used the fact that cross ratios are preserved by global M\"{o}bius transformations. To find $\dot{z}'_m -\dot{z}_m$, we know there is $\mu_m \in \mathbb{C}$ by lemma \ref{lem:infmob} such that
	\begin{align*}
	\left(\begin{array}{c}
	\dot{z}'_m -\dot{z}_m \\ 0
	\end{array}\right) &= 		
	(\Phi_{\mathfrak{uvs}} - \Phi_{\mathfrak{vut}} ) \left(\begin{array}{c}
	z_m \\ 1
	\end{array}\right) - \mu_{m} \left(\begin{array}{c}
	z_m \\ 1
	\end{array}\right) \\
	&= \frac{\lambda_{\mathfrak{uv}}}{\omega(e_{\mathfrak{uv}})} \left( \begin{array}{c}
	z_i(z_m - z_i) \\ (z_m- z_i)
	\end{array} \right) - \mu_{m} \left(\begin{array}{c}
	z_m \\ 1
	\end{array}\right) 
	\end{align*}
	The identity on the second row yields $\mu_m = \frac{\lambda_{\mathfrak{uv}}}{\omega(e_{\mathfrak{uv}})} (z_m-z_i)$ and hence the first row implies
	\[
	\dot{z}'_m -\dot{z}_m = - \frac{\lambda_{\mathfrak{uv}}}{\omega(e_{\mathfrak{uv}})} (z_m-z_i)^2
	\]
	Substituting it into \eqref{eq:dercross}, we deduce that $\dot{\cratio}^{\dagger}/\cratio^{\dagger} = \lambda$. In particular, $\dot{z}$ is generated by a global infinitesimal M\"{o}bius transformation if and only if $ \lambda\equiv 0 $.  
\end{proof}

\section{Discrete minimal surfaces of Koebe Type} \label{sec:kobe}

We show that discrete minimal surfaces of Koebe type can be constructed from holomorphic quadratic differentials of Koebe type via a Weierstrass representation formula. Recall that a Koebe polyhedral surface is a polyhedral surface with edges tangent to the unit sphere. As discussed in the introduction, discrete minimal surfaces of Koebe type are derived from Steiner's formula with vertex normals as a Koebe polyhedral surface. Recall that given a polyhedral surface $f:V \to \mathbb{R}^3$, its area for each face $\phi=(v_1,v_2,\dots,v_n,v_{n+1}=v_1)$ is the magnitude of
\[
\vec{A}(f)_{\phi} = \sum_j f_j \times f_{j+1}.
\] 
Suppose $N:V \to \mathbb{R}^3$ is a vertex normal which defines a family of parallel surfaces $f^t:= f + tN$. Then
\[
\vec{A}(f^t)_{\phi} = \sum_j f_j \times f_{j+1} + t (\sum_j f_j \times N_{j+1} + N_j \times f_{j+1}) + t^2 (\sum_i N_j \times N_{j+1}) .
\]
By comparing to Steiner's formula (Eq. \ref{eq:steiner}), we deduce that the mean curvature $H_{\phi}$ is proportional to
\[
\sum_j f_j \times N_{j+1} + N_j \times f_{j+1}.
\]

\begin{definition}[Bobenko-Pottmann-Wallner \cite{Bobenko2010a}]\label{def:koebemin}
	A polyhedral surface $f:V \to \mathbb{R}^3$ with a Koebe polyhedron $N:V \to \mathbb{R}^3$ as edge offset is a discrete minimal surface of Koebe type if its mean curvature with respect to $N$ defined by Steiner's formula vanishes, i.e. for every face $\phi=(v_1,v_2,\dots,v_n,v_{n+1}=v_1)$
	\[
	\sum_j f_j \times N_{j+1} + N_j \times f_{j+1} =0.
	\]
\end{definition}

On the other hand, it is known that circle packings are closely related to Koebe polyhedral surfaces. The intersection of a Koebe polyhedral surface with the unit sphere yields a circle packing on the sphere. Taking the pole of the faces induces another Koebe polyhedron and the dual circle packing. Conversely, every circle packing $C$ on the sphere that admits a dual circle packing $C^*$ is the intersection of a Koebe polyhedron with the unit sphere. For such a circle packing $C$ with the combinatorics $G$, we denote $N_{C}: F(G) \to \mathbb{R}^3$ the vertices of the Koebe polyhedral surface whose intersection with the sphere is the circle packing $C$ and $N_{C^*}: V(G) \to \mathbb{R}^3$ the dual Koebe polyhedral surface whose intersection with the sphere is the dual circle packing $C^*$. In the following we consider circle packings that correspond to Koebe polyhedral surfaces. This property is always satisfied if $G$ is triangulated. 

\begin{theorem}\label{thm:koebeweierstrass}
	Suppose $C$ is a circle packing with the combinatorics of a simply connected surface $G$ that admits a dual circle packing. We denote $MG$ the medial graph and $z:V(MG) \to \mathbb{C}$ the tangency points of the circles. We assume $\lambda : E_{int}(G) \to \mathbb{R}$ is a discrete holomorphic quadratic differential of Koebe type. Then there exists a realization of the dual mesh $\mathcal{F}: F(G) \to \mathbb{C}^3$ such that
	\begin{align}\label{eq:weierstrass}
	d\mathcal{F}(e_{\mathfrak{uv}}) := \mathcal{F}_{\text{left}(e_{\mathfrak{uv}})}-\mathcal{F}_{\text{right}(e_{\mathfrak{uv}})}= \frac{\lambda_{\mathfrak{uv}}}{\omega(e_{\mathfrak{uv}})} \left( \begin{array}{c}
	1-z_{i}^2 \\ \mathbf{i}(1+z^2_{i}) \\ 2 z_{i} 
	\end{array} \right)
	\end{align}
	where $\text{left}(e_{\mathfrak{uv}})$ and right$(e_{\mathfrak{uv}})$ denote the left and the right face of $e_{\mathfrak{uv}}$ while $z_i$ is the tangency point of circles $C_{\mathfrak{u}}$ and $C_{\mathfrak{v}}$. Furthermore 
	\begin{enumerate}
		\item $\Re \mathcal{F}$ is a discrete minimal surface of Koebe type with vertex normal $N_{C}$ and
		\item $\Im \mathcal{F}$ is a reciprocal parallel mesh of $N_{C^*}$, i.e. $\Im \mathcal{F}$ has corresponding edges parallel to those of $N_{C^*}$ and is defined on the dual graph of $N_{C^*}$.
	\end{enumerate} The corresponding edge lengths of $\Re \mathcal{F}$ and $\Im \mathcal{F}$ are the same.
	
	Conversely, given a discrete minimal surface of Koebe type or a reciprocal parallel mesh of $N_{C^*}$, there exists a unique discrete holomorphic quadratic differential of Koebe type satisfying \eqref{eq:weierstrass}.
\end{theorem}

The idea is similar to Proposition \ref{thm:Weierstrass}. We rewrite the equations of the holomorphic quadratic differentials as the closeness of a $\mathbb{C}^3$-valued 1-form. The following two lemmas can be verified directly.      

\begin{lemma}\label{lem:weier}
	\[
	\begin{cases}
	\sum_\mathfrak{v}  \frac{\lambda_{\mathfrak{uv}}}{\omega(e_{\mathfrak{uv}})} &=0 \\
	\sum_\mathfrak{v}  \frac{\lambda_{\mathfrak{uv}}}{\omega(e_{\mathfrak{uv}})} z_{i} &=0 \\
	\sum_\mathfrak{v}  \frac{\lambda_{\mathfrak{uv}}}{\omega(e_{\mathfrak{uv}})} z_{i}^2 &= 0
	\end{cases}
	\iff	
	\sum_\mathfrak{v}  \frac{\lambda_{\mathfrak{uv}}}{\omega(e_{\mathfrak{uv}})} \left( \begin{array}{c}
	1-z_{i}^2 \\ \mathbf{i}(1+z^2_{i}) \\ 2 z_{i} 
	\end{array} \right) =0
	\]
\end{lemma}

\begin{lemma}\label{lem:sterographic}
	Let $\sigma: \mathbb{C} \to \mathbb{S}^2 \subset \mathbb{R}^3$ be the stereographic projection:
	\[
	\sigma(z) = \frac{1}{1+|z|^2} \left( \begin{array}{c} 2 \Re z \\ 2 \Im z \\ |z|^2 - 1
	\end{array} \right).
	\]
	Then for any $v \in T_z \mathbb{C}$
	\[
	d\sigma_z (v) = \frac{2 |v|^2 }{(1+ |z|^2)^2} \Re(  \frac{1}{v} \left( \begin{array}{c}
	1-z^2 \\ \mathbf{i}(1+z^2) \\ 2 z
	\end{array} \right) )  \in T_{\sigma(z)} \mathbb{S}^2  \subset \mathbb{R}^3
	\]
\end{lemma}

\begin{proof}[Proof of Theorem \ref{thm:koebeweierstrass}]
	By Lemma \ref{lem:weier}, the right hand side of Eq. \eqref{eq:weierstrass} defines a closed dual 1-form on $G$. Since $G$ is simply connected, the 1-form is exact and hence $\mathcal{F}: F(G) \to \mathbb{C}^3$ exists. 
	
	We then examine $\Re \mathcal{F}$ and $\Im \mathcal{F}$ respectively. Lemma \ref{lem:sterographic} implies that $\Re  d\mathcal{F}(e_{\mathfrak{uv}})$ is parallel to $d\sigma_z (\frac{\omega(e_{\mathfrak{uv}})}{\lambda_{\mathfrak{uv}}})$. Because $\lambda_{\mathfrak{uv}}$ is real and $\omega(e_{\mathfrak{uv}})$ is tangent to $C_\mathfrak{u}$ and $C_\mathfrak{v}$ at the tangency point $z_i$, it yields that $\Re  d\mathcal{F}(e_{\mathfrak{uv}})$ is tangent to $\sigma( C_\mathfrak{u})$ and $\sigma(C_\mathfrak{v})$ at $\sigma(z_i)$. We thus deduce that $N_C$ is parallel to $\Re \mathcal{F}$:
	\[
	\Re (\mathcal{F}_{\text{left}(e_{\mathfrak{uv}})}-\mathcal{F}_{\text{right}(e_{\mathfrak{uv}})}) \parallel (N_{C,\text{left}(e_{\mathfrak{uv}})}-N_{C,\text{right}(e_{\mathfrak{uv}})}).
	\]
	 On the other hand $\mathbf{i} \omega(e_{\mathfrak{uv}})$ is tangent to both $C_{ijk}$ and $C_{imn}$. Lemma \ref{lem:sterographic} again leads to
	 \[
	 \Im (\mathcal{F}_{\text{left}(e_{\mathfrak{uv}})}-\mathcal{F}_{\text{right}(e_{\mathfrak{uv}})}) \parallel (N_{C^*,\mathfrak{v}}-N_{C^*,\mathfrak{u}}).
	 \]
	 Hence $\Im \mathcal{F}$ is a so-called reciprocal parallel mesh of $N_{C^*}$ because their combinatorics are dual to each other and the corresponding edges are parallel. 
	
	Furthermore, every edge of $\Re \mathcal{F}$ differs from that of $\Im \mathcal{F}$ by a 90-degree rotation:
	\[
	\Im  d\mathcal{F}(e_{\mathfrak{uv}})= \Re( \frac{\lambda_{\mathfrak{uv}}}{\mathbf{i} \omega(e_{\mathfrak{uv}})} \left( \begin{array}{c}
	1-z_{i}^2 \\ \mathbf{i} (1+z^2_{i}) \\ 2 z_{i} 
	\end{array} \right) ) = \sigma(z_i) \times \Re  d\mathcal{F}(e_{\mathfrak{uv}})
	\]
	since the stereographic projection $\sigma$ is conformal and $| \sigma(z_i) |=1$.
	
	 It remains to show that $\Re \mathcal{F}$ has vanishing mean curvature in order to be a discrete minimal surface of Koebe type (Definition \ref{def:koebemin}). We consider a face $\phi=(v_1,v_2,\dots,v_n,v_{n+1}=v_1)$ of $\Re \mathcal{F}$. We denote $\sigma(z_{jj+1})$ the point where the line $N_{C,j} N_{C,j+1}$ touches at the sphere. Then its mean curvature defined by Steiner's formula is proportional to 
	\begin{align}\label{eq:mixedarea}
	\begin{split}
	\sum_j \Re \mathcal{F}_j \times N_{C,j+1} + N_{C,j} \times \Re \mathcal{F}_{j+1}  =& -2\sum_j (\Re \mathcal{F}_{j+1}-\Re \mathcal{F}_{j}) \times \frac{ N_{C,j+1} + N_{C,j}}{2} \\
	=& -2\sum_j (\Re \mathcal{F}_{j+1}-\Re \mathcal{F}_{j}) \times \sigma(z_{jj+1}) \\
	=& 2\sum_j \Im \mathcal{F}_{j+1}-\Im \mathcal{F}_{j}\\
	=& 0
	\end{split}
	\end{align}
	since both  $(N_{C,j+1}+ N_{C,j})/2$ and $\sigma(z_{jj+1})$ lie on the line $N_{C,j} N_{C,j+1}$ and  hence $(N_{C,j+1}+ N_{C,j})/2 - \sigma(z_{jj+1})$ is parallel to $\Re \mathcal{F}_{j+1}-\Re \mathcal{F}_{j}$. 
	
	The converse argument is similar to the proof of Proposition \ref{thm:Weierstrass} as in \cite{Lam2016}. To see this, suppose $f$ is a discrete minimal surface of Koebe type with vertex normal $N$. Since $N$ is a Koebe polyhedral surface, its intersection with the sphere yields a circle packing with the combinatorics $G$ (which admits a dual circle packing). We consider the stereographic projection of the circle packing in the plane  and denote $z:V(MG)\to \mathbb{C}$ the tangency points. 
	
	We claim that there exists $\mathcal{F}: F(G) \to \mathbb{C}^3$ in the form of \eqref{eq:weierstrass} for some $\lambda:E_{int} \to \mathbb{R}$ such that $f= \Re \mathcal{F}$ and hence $\lambda$ is the desired holomorphic quadratic differential of Koebe type. Pick an edge $\{\mathfrak{uv}\} \in E(G)$ and write $z_i$ the tangency point of $C_{\mathfrak{u}}$ and $C_{\mathfrak{v}}$. Notice that $f:F(G) \to \mathbb{R}^3$ is parallel to $N=N_C$ by definition and hence parallel to $d\sigma_{z_i}(\omega(e_{\mathfrak{uv}}))$ by Lemma \ref{lem:sterographic}. We deduce that there exists $\lambda:E_{int} \to \mathbb{R}$ such that
	\[
	f_{\text{left}(e_{\mathfrak{uv}})}-f_{\text{right}(e_{\mathfrak{uv}})}= \Re( \frac{\lambda_{\mathfrak{uv}}}{\omega(e_{\mathfrak{uv}})} \left( \begin{array}{c}
	1-z_{i}^2 \\ \mathbf{i}(1+z^2_{i}) \\ 2 z_{i} 
	\end{array} \right))
	\]
	In particular, for every $\mathfrak{u} \in V(G)$
	\[
	\Re( \sum_{\mathfrak{v}}  \frac{\lambda_{\mathfrak{uv}}}{\omega(e_{\mathfrak{uv}})} \left( \begin{array}{c}
	1-z_{i}^2 \\ \mathbf{i}(1+z^2_{i}) \\ 2 z_{i} 
	\end{array} \right)) = \sum_{\mathfrak{v}} (f_{\text{left}(e_{\mathfrak{uv}})}-f_{\text{right}(e_{\mathfrak{uv}})})=0
	\]
	Furthermore, reversing the argument in \eqref{eq:mixedarea} yields
		\begin{align*}
				\Im (\sum_{\mathfrak{v}}  \frac{\lambda_{\mathfrak{uv}}}{\omega(e_{\mathfrak{uv}})} \left( \begin{array}{c}
			1-z_{i}^2 \\ \mathbf{i}(1+z^2_{i}) \\ 2 z_{i} 
			\end{array} \right)) &=  \sum_{\mathfrak{v}} \sigma(z_i) \times (f_{\text{left}(e_{\mathfrak{uv}})}-f_{\text{right}(e_{\mathfrak{uv}})}) \\
			&= \sum_{\mathfrak{v}} \frac{(N_{\text{left}(e_{\mathfrak{uv}})}+N_{\text{right}(e_{\mathfrak{uv}})})}{2} \times (f_{\text{left}(e_{\mathfrak{uv}})}-f_{\text{right}(e_{\mathfrak{uv}})})\\
			&=0
		\end{align*}
		since $f$ has vanishing mean curvature over each face. Thus for every $\mathfrak{u} \in V(G)$
		\[
		\sum_{\mathfrak{v}}  \frac{\lambda_{\mathfrak{uv}}}{\omega(e_{\mathfrak{uv}})} \left( \begin{array}{c}
		1-z_{i}^2 \\ \mathbf{i}(1+z^2_{i}) \\ 2 z_{i} 
		\end{array} \right))  =0
		\]
		and hence $\lambda$ is a discrete holomorphic quadratic differential of Koebe type by Lemma \ref{lem:weier}.
		
\end{proof}

\begin{figure} 
	\includegraphics[width=0.5\textwidth]{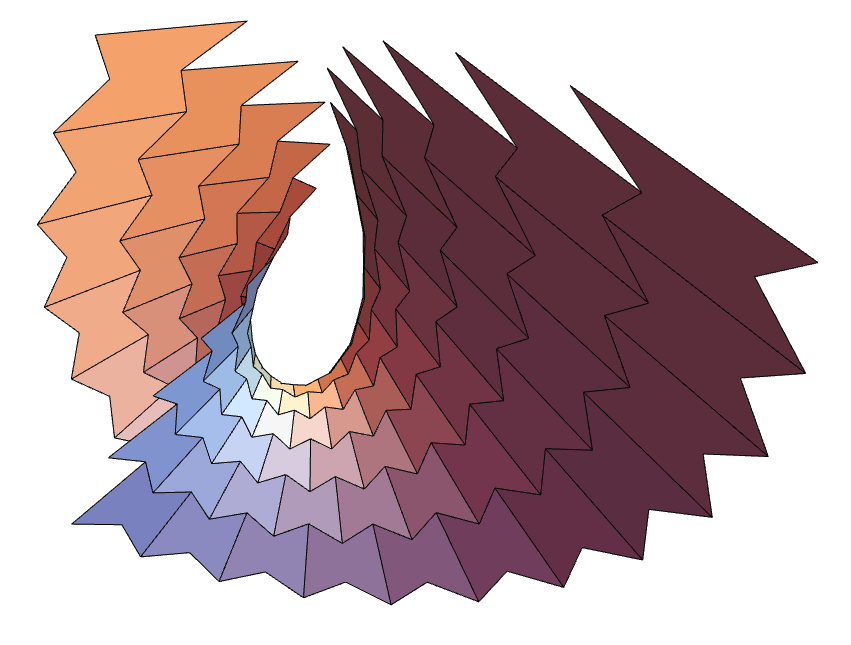}
	\caption{A discrete minimal surface of Koebe type from an infinitesimal deformation of the hexagonal circle packing}
	\label{fig:enneper}
\end{figure}

\section{From Kobe type to general type} \label{sec:unif}

Infinitesimal deformations of circle packings yield two types of holomorphic quadratic differentials, which lead to two kinds of discrete minimal surfaces. In this section, we prove the main theorem that a discrete minimal surface of Koebe type can be extended to a minimal surface of general type naturally and the extension depends only on the triangulation of the medial graph. The proof is straightforward using the $\mbox{sl}(2,\mathbb{C})$-representation for infinitesimal deformations of circle packings in Proposition \ref{prop:infmobgen} and Theorem \ref{prop:infmobkoebe}.

\begin{theorem}\label{thm:kobegeneral}
	Suppose $z:V(MG) \to \mathbb{C}$ is the tangency points of a circle packing with the combinatorics of a triangulated disk $G$ and $TMG$ is a triangulation of the medial graph $MG$ with $V(TMG)=V(MG)$. Then there is a one-to-one correspondence between holomorphic quadratic differentials of Koebe type $\lambda:E_{int}(G)\to \mathbb{R}$ and those of the general type $q:E_{int}(TMG) \to \mathbb{R}$.
	
	Furthermore, the holomorphic quadratic differentials respectively define realizations $\mathcal{F}: F(G) \to \mathbb{C}^3$ and $\hat{\mathcal{F}}:F(TMG) \to \mathbb{C}^3$ via the Weierstrass representation formulas (Proposition \ref{thm:Weierstrass} and Theorem \ref{thm:koebeweierstrass}). We have
	\[
	\hat{\mathcal{F}}|_{F(G)} = \mathcal{F}
	\]
	up to a translation. 
\end{theorem}

\begin{proof}
	Every infinitesimal deformation of a circle packing induces a vector field on the tangency points $\dot{z}:V(MG)\to \mathbb{C}$. Lemma \ref{lem:infmob} then implies that there are unqiue $\Phi:F(G) \to sl(2,\mathbb{C})$ and $\tilde{\Phi}: F(TMG) \to sl(2,\mathbb{C})$ that describe infinitesimal M\"{o}bius transformations over each face. Proposition \ref{prop:infmobgen} and Proposition \ref{prop:infmobkoebe} further yields that $\Phi$ and $\tilde{\Phi}$ correspond to holomorphic quadratic differential of Koebe type $\lambda$ and that of general type $q$ respectively:
	\[
	\lambda \xleftrightarrow{\text{Proposition \ref{prop:infmobkoebe}}} \dot{z} \xleftrightarrow{\text{Proposition \ref{prop:infmobgen}}} q
	\]
	
	Secondly, in the notation of Figure \ref{fig:orientation} left, a face $\{\mathfrak{uvs}\}\in F(G)$ has tangency points $z_{i}\in C_{\mathfrak{u}} \cap C_{\mathfrak{v}}$, $z_{j}\in C_{\mathfrak{v}} \cap C_{\mathfrak{s}}$, $z_{k}\in C_{\mathfrak{s}} \cap C_{\mathfrak{u}}$ which form a face $\{ijk\} \in F(TMG)$ on the other hand. It yields an injection $F(G) \hookrightarrow F(TMG)$. By construction,
	\[
	\hat{\Phi}_{ijk} = \Phi_{\mathfrak{uvs}}
	\]
	and we have $\hat{\Phi}|_{F(G)} = \Phi$. 
	
	We consider an interior edge $\{\mathfrak{uv}\} \in E(G)$ and the two neighboring triangles $\{\mathfrak{uvs}\}$, $\{\mathfrak{vut}\} \in F(G)$ (see Figure \ref{fig:orientation} left). The five tangency points are denoted by $z_i,z_j,z_k,z_m,z_n$. We write the other points of tangency on $C_\mathfrak{u}$ connecting to $z_i$ as $z_n=z_1,z_2,\dots z_p=z_j$, i.e. $\{i1\},\{i2\}\dots\{ip\} \in E(TMG)$. Then
	\[
	\frac{\lambda_{\mathfrak{uv}}}{\omega(e_{\mathfrak{uv}})}\left(\begin{array}{cc}
	z_{i} & -z^2_{i} \\1 & -z_{i}
	\end{array} \right) = \Phi_{\mathfrak{uvs}} - \Phi_{\mathfrak{vut}}= \hat{\Phi}_{ijk} - \hat{\Phi}_{imn}= 
\sum_{r=1}^{p}\frac{q_{ij}}{z_r - z_i} \left( \begin{array}{cc}
	\frac{z_i + z_r}{2} & -z_i z_r \\
	1  & \frac{z_i + z_r}{2} 
	\end{array} \right) 
	\]
	
	It implies
	\[
	\mathcal{F}_{\mathfrak{uvs}} - \mathcal{F}_{\mathfrak{vut}} =  \frac{\lambda_{\mathfrak{uv}}}{\omega(e_{\mathfrak{uv}})} \left( \begin{array}{c}
	1-z_{i}^2 \\ \mathbf{i}(1+z^2_{i}) \\ 2 z_{i} 
	\end{array} \right)  = \sum_{r=1}^{p} \frac{q_{ir}}{z_r - z_i} \left( \begin{array}{c} 1-z_iz_r \\ \mathbf{i} (1+z_iz_r) \\ z_i+z_r \end{array} \right) = \hat{\mathcal{F}}_{ijk} - \hat{\mathcal{F}}_{imn}. 
	\]
	Hence we deduce that
	\[
	\hat{\mathcal{F}}|_{F(G)} = \mathcal{F}
	\]
	up to translation.
\end{proof}

\section{Discrete harmonic functions} \label{sec:har}

In this section, we divert our attention from discrete minimal surfaces to another equivalent parametrization of circle packings: vertex rotation. Using this approach, infinitesimal deformations of a circle packing correspond to discrete harmonic functions with respect to the cotangent Laplacian. We discuss its relation to the known approach of discrete harmonic functions using the change in radii of circles. 

\subsection{Vertex rotation} Luo \cite{Luo2004} introduced a notion of discrete conformality based on the edge lengths of a triangle mesh -- vertex scaling. Motivated by his approach, we introduce vertex rotation to describe deformations of circle packings. Vertex rotation was first considered on circle patterns where neighboring circles intersect \cite{Lam2015a}. 

\begin{definition}
	Two triangle meshes $z,\tilde{z}:V \to \mathbb{C}$ differ by vertex rotation if there exists $\alpha:V \to \mathbb{R}$ such that
	\begin{align*}
	\frac{\tilde{z}_j - \tilde{z}_i}{|\tilde{z}_j - \tilde{z}_i|} = e^{i(\alpha_i + \alpha_j)  }  \frac{z_j - z_i}{|z_j - z_i|}
	\end{align*}
	for every interior edge $\{ij\}$.
\end{definition}
With an analogue of Ptolemy's theorem, we extend vertex rotation to medial graphs induced from circle packings.

\begin{lemma}\label{lem:sixthpair}
	Let $z_1,z_2,z_3,z_4$ and $\tilde{z}_1,\tilde{z}_2,\tilde{z}_3,\tilde{z}_4$ be two collections of con-cyclic points in the plane $\mathbb{C}$. Suppose $\alpha_1,\alpha_2,\alpha_3,\alpha_4 \in \mathbb{R}$ satisfy
	\begin{align}\label{eq:lemrot}
	\frac{\tilde{z}_j - \tilde{z}_i}{|\tilde{z}_j - \tilde{z}_i|} = e^{i(\alpha_i + \alpha_j)  }  \frac{z_j - z_i}{|z_j - z_i|}
	\end{align}
	for any five pairs of points $i\neq j$. Then it holds as well for the remaining pair. 
\end{lemma}
\begin{proof}
	We assume Eq.\eqref{eq:lemrot} holds for pairs of points: $12,23,34,13,14$. We show that it holds for the pair $24$ as well. Since $z_i$'s are con-cyclic, their cross ratios are real. The sign of the cross ratio changes if we permute two neighboring vertices. Hence
	\begin{align*}
	\frac{\cratio(z_1,z_3,z_2,z_4)}{|\cratio(z_1,z_3,z_2,z_4)|} =-\frac{\cratio(z_1,z_2,z_3,z_4)}{|\cratio(z_1,z_2,z_3,z_4)|} = \pm 1
	\end{align*}
	where
	\[
	\cratio(z_1,z_2,z_3,z_4) = \frac{(z_1-z_2)(z_3-z_4)}{(z_2-z_3)(z_4-z_1)}.
	\]
	We thus have
	\begin{align*}
	\frac{z_3 - z_1}{|z_3 - z_1|} \frac{z_4 - z_2}{|z_4 - z_2|} = \frac{z_2 - z_1}{|z_2 - z_1|} \frac{z_4 - z_3}{|z_4 - z_3|}
	\end{align*}
	and similarly for $\tilde{z}$. So
	\begin{align*}
	\frac{\tilde{z}_3 - \tilde{z}_1}{|\tilde{z}_3 - \tilde{z}_1|} \frac{\tilde{z}_4 - \tilde{z}_2}{|\tilde{z}_4 - \tilde{z}_2|} &= e^{i(\alpha_1+\alpha_2+\alpha_3+\alpha_4)}\frac{z_3 - z_1}{|z_3 - z_1|} \frac{z_4 - z_2}{|z_4 - z_2|}  \\
	\frac{\tilde{z}_4 - \tilde{z}_2}{|\tilde{z}_4 - \tilde{z}_2|} &= e^{i(\alpha_2+\alpha_4)} \frac{z_4 - z_2}{|z_4 - z_2|}
	\end{align*}
	
\end{proof}

We consider \emph{edge flipping} on triangle meshes: Suppose two triangles $\{ijk\}$ and $\{jil\}$ share a common edge $\{ij\}$. We switch the edge $\{ij\}$ for $\{kl\}$ and produce two triangles $\{klj\}$ and $\{lki\}$. A fact on edge flipping is used \cite{Hurtado1999}:

\begin{lemma} \label{lem:edgeflip}
	Any two triangulations of a polygon are connected via edge flipping.
\end{lemma}

\begin{proposition} \label{thm:verrotcircle}
	Suppose $C$ is a circle packing of the combinatorics of a triangle mesh $G$. We denote $z: V(MG) \to \mathbb{C}$ the medial graph formed by the tangency points of the circles. Then a realization $\tilde{z}: V(MG) \to \mathbb{C}$ is the tangency points of another circle packing $\tilde{C}$ if and only if $z$ and $\tilde{z}$ differ by vertex rotation, i.e.
	there exists $\alpha : V(MG) \to \mathbb{R}$ such that
	\begin{align*}
	\frac{\tilde{z}_j - \tilde{z}_i}{|\tilde{z}_j - \tilde{z}_i|} = e^{i(\alpha_i + \alpha_j)  }  \frac{z_j - z_i}{|z_j - z_i|}
	\end{align*}
	for every vertices $i,j$ within the same face of the medial graph $MG$.
\end{proposition}

\begin{proof}
	The medial graph $MG$ has two types of faces. One type of them corresponds to the triangular faces of $G$. The other type corresponds to the vertices of $G$ which are cyclic polygons. We triangulate each  polygon without introducing new vertices. In this way, we obtain a triangle mesh $TMG$ with $V(TMG)= V(MG)$. Furthermore, $z$ and $\tilde{z}$ are the intersection points of two circle patterns with the same pattern structure. By \cite[Theorem 2.6]{Lam2015a}, there exists  $\alpha : V(TMG) \to \mathbb{R}$ such that
	\begin{align*}
	\frac{\tilde{z}_j - \tilde{z}_i}{|\tilde{z}_j - \tilde{z}_i|} = e^{i(\alpha_i + \alpha_j)  }  \frac{z_j - z_i}{|z_j - z_i|}
	\end{align*}
	for every $\{ij\} \in E(TMG)$. Lemma \ref{lem:sixthpair} and \ref{lem:edgeflip} imply $\alpha$ is independent of the triangulation.
\end{proof}

\subsection{Discrete harmonic functions}

Various graph Laplacians have been proposed with different edge weights. The cotangent Laplacian is a graph Laplacian usually defined on a triangle mesh with cotangent weights induced from edge lengths:
\begin{definition}\label{def:harmonic}
	Suppose $z:V \to \mathbb{C}$ is a realization of a triangle mesh. A function $u:V \to \mathbb{R}$ is harmonic with respect to the cotangent Laplacian if for every interior vertex $i$
	\[
	\sum_j (\cot \angle jki + \cot \angle ilj) (u_j - u_i) =0
	\]
	where $\{ijk\},\{jil\}$ are two neighboring faces sharing edge $\{ij\}$ (see Figure \ref{fig:orientation}).
\end{definition}

Discrete harmonic functions were introduced on the square lattice by Ferrand \cite{Ferrand1944} and Duffin \cite{Duffin1956} in terms of discrete Cauchy-Riemann equations. This notion was later generalized to triangular meshes and led to the cotangent Laplacian, which is central to linear discrete complex analysis \cite{Dmitry2011}. A convergence result of discrete harmonic functions was discussed in \cite{Skopenkov2013}. 

The cotangent Laplacian can be generalized to meshes where all faces are cyclic. To apply the cotangent formula, one needs to triangulate the mesh by adding diagonals to all cyclic faces. However, the cotangent weight on the diagonals vanishes because neighboring faces are inscribed in the same circle and the sum of opposite angles is $\pi$. This observation implies that the cotangent weights are defined on the original edges of the given mesh and are independent of the triangulation.

\begin{lemma}
	Suppose  the medial graph $MG$ is subdivided into a triangle mesh $TMG$ with the vertex set $V(MG)= V(TMG)$ and edge set $E(MG) \subset E(TMG)$. Using the notation in Fig. \ref{fig:orientation}, we have the following:	
	\begin{enumerate}
		\item If $\{ij\} \in  E(TMG)-E(MG)$, then
		\[
		\cot \angle jki + \cot \angle ilj =0
		\]
		\item If $\{ij\} \in  E(MG)$, then $\{ij\}$ connects some tangency points  $C_{\mathfrak{u}} \cap C_{\mathfrak{v}}$ and $C_{\mathfrak{s}} \cap C_{\mathfrak{v}}$. We have
		\begin{align*}
		\cot \angle jki + \cot \angle ilj = \frac{R_\mathfrak{v}}{R_{ijk}} + \frac{R_{ijk}}{R_\mathfrak{v}} 
		\end{align*} 
		where $R_{\mathfrak{v}}$ and $R_{ijk}$ are the radii of circles $C_{\mathfrak{v}}$ and $C_{ijk}$.
	\end{enumerate}
	In particular, the cotangent weights are well defined on the medial graph $MG$ and are independent of the triangulation.
\end{lemma}
\begin{proof}
	If $\{ij\} \in  E(TMG)-E(MG)$, then $z_i,z_j,z_k,z_l$ are con-cyclic points, we either have the sum of angles
	\begin{align*}
	\angle jki + \angle ilj = \pi \quad \text{or} \quad 0
	\end{align*}
	which both leads to $ \cot \angle jki + \cot \angle ilj =0$.
	If $\{ij\} \in  E(MG)$, then the circles $C_{ijk}$ and $C_{\mathfrak{v}}$ intersect orthogonally. We have $\angle jki = \pi/2 - \angle ilj$ and $\cot \angle jki = R_{ijk}/R_{\mathfrak{v}}$.
\end{proof}

In the proposition below, the relation between (1) and (2) below has been shown by Glickenstein \cite{Glickenstein2011} while the one between (1) and (3) is similar to the case for circle patterns \cite{Lam2015a}.
\begin{proposition} \label{prop:infinrot}
	Using the notation in Fig. \ref{fig:orientation}, there is a one-to-one correspondence between
	\begin{enumerate}
		\item infinitesimal deformations of a circle packing up to a Euclidean motion;
		\item discrete harmonic functions $\sigma: V(G) \to \mathbb{R}$ satisfying for every $\mathfrak{u} \in V_{int}(G)$
		\begin{align*}
		\sum_{\mathfrak{v}} \frac{R_{ijk}+ R_{imn}}{R_{\mathfrak{u}} + R_{\mathfrak{v}}} (\sigma_{\mathfrak{v}} - \sigma_{\mathfrak{u}}) =0
		\end{align*} 
		where $\sigma$ describes the change of radii $\dot{R} = \sigma R$;
		\item discrete harmonic functions on the medial graph $\alpha: V(MG) \to \mathbb{R}$ in the sense of the cotangent Laplacian, i.e. for every $i \in V(MG)$
		\begin{align*}
		0=&(\frac{R_{\mathfrak{v}}}{R_{ijk}}+ \frac{R_{ijk}}{R_{\mathfrak{v}}})(\alpha_{j}-\alpha_{i}) + (\frac{R_{\mathfrak{u}}}{R_{ijk}}+ \frac{R_{ijk}}{R_{\mathfrak{u}}})(\alpha_{k}-\alpha_{i}) \\ &+ (\frac{R_{\mathfrak{v}}}{R_{imn}}+ \frac{R_{imn}}{R_{\mathfrak{v}}})(\alpha_{n}-\alpha_{i}) + (\frac{R_{\mathfrak{u}}}{R_{imn}}+ \frac{R_{imn}}{R_{\mathfrak{u}}})(\alpha_{m}-\alpha_{i}) 
		\end{align*}
		Here $\alpha$  is induced from infinitesimal vertex rotation on the medial graph.
	\end{enumerate}
	The two types of harmonic functions in (2) and (3) are related via
	\[
	\alpha_{i} = \eta_{\mathfrak{uvs}} - \frac{R_{ijk}}{R_{\mathfrak{u}} + R_{\mathfrak{v}}} (\sigma_{\mathfrak{v}} - \sigma_{\mathfrak{u}})
	\]
	where $\eta:F(G) \to \mathbb{R}$ is a harmonic conjugate of $\sigma$.		
\end{proposition}
\begin{proof}
	
	Here we mainly show the correspondence between the two types of discrete harmonic functions, though along the way the statements $(1) \implies (3)$ and $(1)\implies (2)$ are proven. Every circle packing can be described by the centers $c_{\mathfrak{u}} \in \mathbb{C}$ and the radii $R_{\mathfrak{u}}$ of the circles satisfying for every $\mathfrak{uv} \in E(G)$: 
	\begin{align*}
	|c_{\mathfrak{v}} - c_{\mathfrak{u}}|^2 &= (R_{\mathfrak{u}} + R_{\mathfrak{u}})^2
	\end{align*}
	A first order change of the centers $\dot{c}$ and the radii $\dot{R}$ describe an infinitesimal deformation of a circle packing if and only if 
	\begin{align*}
	\langle \dot{c}_{\mathfrak{v}} - \dot{c}_{\mathfrak{u}}, c_{\mathfrak{v}} - c_{\mathfrak{u}}  \rangle  &= (\dot{R}_{\mathfrak{u}} + \dot{R}_{\mathfrak{u}}) (R_{\mathfrak{u}} + R_{\mathfrak{u}})
	\end{align*}
	or equivalently
	\begin{align*}
	\dot{c}_{\mathfrak{v}} - \dot{c}_{\mathfrak{u}}
	&=(\frac{\dot{R}_{\mathfrak{u}} + \dot{R}_{\mathfrak{v}}}{R_{\mathfrak{u}} + R_{\mathfrak{v}}} + \mathbf{i} \alpha_i ) (c_{\mathfrak{v}} - c_{\mathfrak{u}}) = (\frac{ \sigma_{\mathfrak{u}} R_{\mathfrak{u}} + \sigma_{\mathfrak{v}} R_{\mathfrak{v}}}{R_{\mathfrak{u}} + R_{\mathfrak{v}}} + \mathbf{i} \alpha_i ) (c_{\mathfrak{v}} - c_{\mathfrak{u}})
	\end{align*}
	for some $\alpha_i \in \mathbb{R}$ and $\sigma:= \dot{R}/R$. Here $\{i\}$ is a vertex of the medial graph which corresponds to the edge $\{\mathfrak{uv}\}$. We claim that both $\sigma$ and $\alpha$ are discrete harmonic functions.
	
	For a triangle $\{ \mathfrak{uvs}\} \in F(G)$, we denote the tangency points $z_{i}\in C_{\mathfrak{u}} \cap C_{\mathfrak{v}}$, $z_{j}\in C_{\mathfrak{v}} \cap C_{\mathfrak{s}}$, $z_{k}\in C_{\mathfrak{s}} \cap C_{\mathfrak{u}}$ and $R_{ijk}$ the radius of circle $C_{ijk}$ through $z_i,z_j,z_k$. Since $C_{ijk}$ and $C_{\mathfrak{u}}$ are orthogonal
	\[
	z_j- z_i = R_{\mathfrak{v}} (\frac{c_{\mathfrak{v}}-c_{\mathfrak{u}}}{R_{\mathfrak{u}}+R_{\mathfrak{v}}}+ \frac{c_{\mathfrak{s}}-c_{\mathfrak{v}}}{R_{\mathfrak{s}}+R_{\mathfrak{v}}})= \mathbf{i} R_{ijk} (\frac{c_{\mathfrak{v}}-c_{\mathfrak{u}}}{R_{\mathfrak{v}}+R_{\mathfrak{v}}} - \frac{c_{\mathfrak{s}}-c_{\mathfrak{v}}}{R_{\mathfrak{s}}+R_{\mathfrak{v}}}).
	\]
	and similarly for $z_k - z_j, z_i - z_k$. Hence
	\begin{align*}
	0=& (\dot{c}_{\mathfrak{v}} - \dot{c}_{\mathfrak{u}}) + (\dot{c}_{\mathfrak{s}} - \dot{c}_{\mathfrak{v}}) +(\dot{c}_{\mathfrak{u}} - \dot{c}_{\mathfrak{s}}) \\
	=& \phantom{+} \sigma_{\mathfrak{u}} R_{\mathfrak{u}}  (\frac{c_{\mathfrak{u}}-c_{\mathfrak{s}}}{R_{\mathfrak{u}}+R_{\mathfrak{s}}}+ \frac{c_{\mathfrak{v}}-c_{\mathfrak{u}}}{R_{\mathfrak{v}}+R_{\mathfrak{u}}}) + \sigma_{\mathfrak{v}} R_{\mathfrak{v}}  (\frac{c_{\mathfrak{v}}-c_{\mathfrak{u}}}{R_{\mathfrak{v}}+R_{\mathfrak{u}}}+ \frac{c_{\mathfrak{s}}-c_{\mathfrak{v}}}{R_{\mathfrak{s}}+R_{\mathfrak{v}}}) + \sigma_{\mathfrak{s}} R_{\mathfrak{s}}  (\frac{c_k-c_{\mathfrak{v}}}{R_{\mathfrak{v}}+R_k}+ \frac{c_{\mathfrak{u}}-c_{\mathfrak{s}}}{R_{\mathfrak{s}}+R_{\mathfrak{u}}}) \\ &+  \mathbf{i} \alpha_{i} (c_{\mathfrak{v}} -  c_{\mathfrak{u}})+  \mathbf{i} \alpha_{j} (c_{\mathfrak{s}} - c_{\mathfrak{v}})+  \mathbf{i} \alpha_{k} (c_{\mathfrak{u}} - c_{\mathfrak{s}}) \\
	=& \phantom{+}  \mathbf{i}(\frac{R_{ijk}}{R_{\mathfrak{u}} + R_{\mathfrak{v}}} ( \sigma_{\mathfrak{v}} -\sigma_{\mathfrak{u}}) +  \alpha_{i} ) (c_{\mathfrak{v}}-c_{\mathfrak{u}}) + \mathbf{i}(\frac{R_{ijk}}{R_{\mathfrak{v}} + R_{\mathfrak{s}}} ( \sigma_{\mathfrak{s}} -\sigma_{\mathfrak{v}}) +  \alpha_{j} ) (c_{\mathfrak{s}}-c_{\mathfrak{v}}) \\&+ \mathbf{i}( \frac{R_{ijk}}{R_{\mathfrak{s}} + R_{\mathfrak{u}}} ( \sigma_{\mathfrak{u}} -\sigma_{\mathfrak{s}}) +  \alpha_{k} ) (c_{\mathfrak{u}}- c_{\mathfrak{s}})
	\end{align*}
	Since $\dim\{\mbox{span}_{\mathbb{R}}\{c_j-c_i,c_k-c_j,c_i-c_k\} \} = 2$ and $(c_j-c_i)+(c_k-c_j)+(c_i-c_k)=0$, there exists $\eta_{\mathfrak{uvs}} \in \mathbb{R}$ such that
	\begin{align}\label{eq:angles}
	\eta_{\mathfrak{uvs}} =& \frac{R_{ijk}}{R_{\mathfrak{u}} + R_{\mathfrak{v}}} ( \sigma_{\mathfrak{v}} -\sigma_{\mathfrak{u}}) +  \alpha_{i} = \frac{R_{ijk}}{R_{\mathfrak{v}} + R_{\mathfrak{s}}} ( \sigma_{\mathfrak{s}} -\sigma_{\mathfrak{v}}) +  \alpha_{j} = \frac{R_{ijk}}{R_{\mathfrak{s}} + R_{\mathfrak{u}}} ( \sigma_{\mathfrak{u}} -\sigma_{\mathfrak{s}}) +  \alpha_{k}
	\end{align}
	Thus, 
	\begin{align*}
	\sum_{\mathfrak{v}} \frac{R_{ijk}+ R_{imn}}{R_{\mathfrak{u}} + R_{\mathfrak{v}}} (\sigma_{\mathfrak{v}} - \sigma_{\mathfrak{u}})= \sum_{\mathfrak{v}} \eta_{\mathfrak{uvs}} - \eta_{\mathfrak{vut}} =0
	\end{align*}
	which proves the claim $(1) \implies (2)$ and $\eta:F \to \mathbb{R}$ is called a harmonic conjugate of $\sigma$. Together with the identity $R^2_{ijk} = R_{\mathfrak{u}} R_{\mathfrak{v}} R_{\mathfrak{s}}/(R_{\mathfrak{u}}+R_{\mathfrak{v}}+R_{\mathfrak{s}})$ one can also rewrite Eq. \eqref{eq:angles} to obtain
	\begin{align*}
	(\frac{R_{\mathfrak{v}}}{R_{ijk}}+ \frac{R_{ijk}}{R_{\mathfrak{v}}})(\alpha_{j}-\alpha_{i})
	=& \frac{ R_{\mathfrak{s}} + R_{\mathfrak{v}}}{R_{\mathfrak{u}} + R_{\mathfrak{v}} + R_{\mathfrak{s}}} (\sigma_{\mathfrak{v}} - \sigma_{\mathfrak{u}}) + \frac{ R_{\mathfrak{u}} + R_{\mathfrak{v}}}{R_{\mathfrak{u}} + R_{\mathfrak{v}} + R_{\mathfrak{s}}} (\sigma_{\mathfrak{v}} - \sigma_{\mathfrak{s}}) 
	\end{align*}
	and deduce $\alpha$ is a discrete harmonic function with respect to the cotangent Laplacian. 
\end{proof}

Combining Proposition \ref{prop:infmobgen}, \ref{prop:infmobkoebe} and \ref{prop:infinrot}, we have three equivalent ways to parameterize infinitesimal deformations of circle packings.

\begin{corollary} \label{thm:equiinfin}
	Each of the following has a one-to-one correspondence with infinitesimal deformations of circle packings up to a trivial motion:
	\begin{enumerate}
		\item Holomorphic quadratic differentials of general type (infinitesimal change in cross ratios as circle patterns). \label{corr:2} 
		\item Holomorphic quadratic differentials of Koebe type (infinitesimal change in cross ratios as circle packings). \label{corr:3} 
		\item Discrete harmonic functions in the sense of the cotangent Laplacian (infinitesimal change in vertex rotation). \label{corr:1} 
	\end{enumerate}
\end{corollary}

In the corollary, trivial infinitesimal deformations for discrete harmonic functions are translation and scaling while those for holomorphic quadratic differentials are M\"{o}bius transformations.

\section*{Acknowledgment}
The author would like to thank Ken Stephenson for pointing out Orick's thesis.

\bibliographystyle{siam}
\bibliography{circpacking}

\end{document}